\documentclass[sn-apa]{sn-jnl}

\usepackage{graphicx}
\usepackage{geometry}
\usepackage{subfigure}
\usepackage{graphicx}
\usepackage{multirow}%
\usepackage{amsmath,amssymb,amsfonts}%
\usepackage{amsthm}%
\usepackage{mathrsfs}%
\usepackage[title]{appendix}%
\usepackage{xcolor}%
\usepackage{textcomp}%
\usepackage{manyfoot}%
\usepackage{booktabs}%
\usepackage{algorithm}%
\usepackage{algorithmicx}%
\usepackage{algpseudocode}%
\usepackage{listings}%
\usepackage{lipsum}
\usepackage{float}
\usepackage{pdflscape}
\usepackage{lineno}
\usepackage{enumerate}
\usepackage{mathtools}
\usepackage{enumitem}
\usepackage{bbm}
\usepackage{comment}
\usepackage{tikz}
\usepackage{makecell}

\newcommand{\C}{\mathbb{C}}
\newcommand{\E}{\mathbb{E}}
\newcommand{\N}{\mathbb{N}}
\renewcommand{\P}{\mathbb{P}}
\newcommand{\Q}{\mathbb{Q}}
\newcommand{\R}{\mathbb{R}}

\newcommand{\V}{\mathbb{V}}
\newcommand{\X}{\mathbb{X}}
\newcommand{\Z}{\mathbb{Z}}
\newcommand{\cA}{{\cal A}}
\newcommand{\cC}{{\cal C}}
\newcommand{\cE}{{\cal E}}
\newcommand{\cG}{{\cal G}}

\newcommand{\cM}{{\cal M}}

\newcommand{\bDelta}{\boldsymbol{\Delta}}
\newcommand{\blambda}{\boldsymbol{\lambda}}
\newcommand{\bLambda}{\boldsymbol{\Lambda}}
\newcommand{\brho}{\boldsymbol{\rho}}
\newcommand{\bSigma}{\boldsymbol{\Sigma}}
\newcommand{\bx}{\boldsymbol{x}}
\newcommand{\bz}{\boldsymbol{z}}
\newcommand{\bA}{\boldsymbol{A}}
\newcommand{\bB}{\boldsymbol{B}}
\newcommand{\bG}{\boldsymbol{G}}
\newcommand{\bJ}{\boldsymbol{J}}
\newcommand{\bR}{\boldsymbol{R}}
\newcommand{\bZ}{\boldsymbol{Z}}
\newcommand{\bgL}{\boldsymbol{\Lambda}}
\theoremstyle{thmstyleone}%
\newtheorem{theorem}{Theorem}

\newtheorem{corollary}{Corollary}

\theoremstyle{thmstyletwo}%
\newtheorem{lemma}{Lemma}%
\newtheorem{example}{Example}%
\newtheorem{remark}{Remark}%

\theoremstyle{thmstylethree}%
\newtheorem{definition}{Definition}
\begin{document}

\title[Article Title]{On the compatibility between the spatial moments and the codomain of a real random field}

\author*[1,2]{\fnm{Xavier} \sur{Emery}}\email{xemery@ing.uchile.cl}

\author[3]{\fnm{Christian} \sur{Lantu\'ejoul}}\email{christian.lantuejoul@minesparis.psl.eu}

\affil[1]{\orgdiv{Department of Mining Engineering}, \orgname{Universidad de Chile}, \orgaddress{\city{Santiago}, \country{Chile}}}

\affil[2]{\orgdiv{Advanced Mining Technology Center}, \orgname{Universidad de Chile}, \orgaddress{\city{Santiago}, \country{Chile}}}

\affil[3]{\orgdiv{Statistiques et Images}, \orgname{Mines ParisTech, PSL University}, \orgaddress{\city{Paris}, \country{France}}}

\abstract{While any symmetric and positive semidefinite mapping can be the non-centered covariance of a Gaussian random field, it is known that these conditions are no longer sufficient when the random field is valued in a two-point set. The question therefore arises of what are the necessary and sufficient conditions for a mapping $\rho: \X \times \X \to \R$ to be the non-centered covariance of a random field with values in a subset ${\cE}$ of $\R$. Such conditions are presented in the general case when ${\cE}$ is a closed subset of the real line, then examined for some specific cases. In particular, if ${\cE}=\R$ or $\Z$, it is shown that the conditions reduce to $\rho$ being symmetric and positive semidefinite. If ${\cE}$ is a closed interval or a two-point set, the necessary and sufficient conditions are more restrictive: the symmetry, positive semidefiniteness, upper and lower boundedness of $\rho$ are no longer enough to guarantee the existence of a random field valued in ${\cE}$ and having $\rho$ as its non-centered covariance. Similar characterizations are obtained for semivariograms and higher-order spatial moments, as well as for multivariate random fields.}


\maketitle


\section{Introduction}

This article deals with fundamental aspects in the modeling of random fields defined on an index set $\X$ and valued in a set $\cE$. Throughout, $\X$ will be an arbitrary finite or infinite set of points, such as a plane, a sphere, or the vertices of a finite graph, to name a few examples. As for the set of destination or \emph{codomain} $\cE$, it will be a subset of $\R$, i.e., the random field is real-valued, which is the most common situation in applications of spatial statistics \citep{chiles_delfiner_2012}, mathematical morphology \citep{Serra1982}, stochastic geometry \citep{Chiu2013}, machine learning \citep{Scholkopf}, and scientific computing \citep{Ghanem}.

A \emph{random field} $Z$ with index set $\X$ and codomain $\cE$ is a collection of random variables defined on the same probability space $(\Omega,{\cA},\P)$. Informally, it can be thought of as a random vector whose components are valued in $\cE$, except that the number of such components (the cardinality of $\X$) can be infinite. For a formal definition, let us endow the real line $\R$ with the usual topology and pose
\begin{equation*}
    \begin{split}
        Z: (\X,\Omega) &\to \cE\\
           (x,\omega) &\mapsto Z(x,\omega) = \omega(x),
    \end{split}
\end{equation*}
where:
\begin{itemize}
    \item for fixed $x$, the mapping $\omega \mapsto Z(x,\omega)$ is a random variable on $(\Omega,{\cA},\P)$; 
    \item for fixed $\omega$, the mapping $x \mapsto Z(x,\omega)$ is called a \emph{realization} (aka a trajectory or a sample path) of the random field; 
    \item $\Omega = {\cE}^\X$, set of all possible realizations of the random field, is called the \emph{sample space};
    \item ${\cA}$ is the Borel $\sigma$-algebra of $\Omega$, called the \emph{event space}, an event being a Borel subset of the sample space;
    \item $\P$ is a probability measure that assigns a probability between $0$ and $1$ to each element of the event space:
    \begin{equation*}
        \begin{split}
            \P: \cA &\to [0,1]\\
            A &\mapsto \P(A)=\int_{\omega \in A} \P({\rm d}\omega),
        \end{split}
    \end{equation*}
    with $\P$ being countably additive and $\P(\Omega)=1$.
\end{itemize}

\medskip

Provided that the random variable $Z(x,\cdot)$ is square integrable with respect to the probability measure $\P$ for all $x \in \X$, the real-valued random field possesses a finite expectation and a finite variance at every point of $\X$, as well as a finite (auto)covariance function and a semivariogram for every pair of 
points. Covariance functions and semivariograms are the fundamental tools in many disciplines dealing with random fields, in particular, they are the building 
blocks of the kriging technique in spatial statistics. The expectation, variance, centered covariance, non-centered covariance, and semivariogram are defined as:
\begin{equation*}
    \begin{split}
        \E(Z(x,\cdot)) &:= \int_{\Omega} Z(x,\omega) \P({\rm d}\omega), \quad x \in \X,\\
        \V(Z(x,\cdot)) &:= \int_{\Omega} Z^2(x,\omega) \P({\rm d}\omega) - \E^2(Z(x,\cdot)), \quad x \in \X,\\
        \C(Z(x,\cdot), Z(y,\cdot)) &:= \int_{\Omega} Z(x,\omega) Z(y,\omega) \P({\rm d}\omega) - \E(Z(x,\cdot)) \, \E(Z(y,\cdot)), \quad x, y \in \X, \\
        \rho(x,y) &:= \int_{\Omega} Z(x,\omega) Z(y,\omega) \P({\rm d}\omega), \quad x, y \in \X,\\
        g(x,y) &:= \frac{1}{2} \int_{\Omega} [Z(x,\omega)-Z(y,\omega)]^2 \P({\rm d}\omega) - \frac{1}{2} \E^2(Z(x,\omega)-Z(y,\omega)), \quad x, y \in \X,
    \end{split}
\end{equation*}
respectively. For the semivariogram to exist, the assumption of square integrability of $Z(x,\cdot)$ at any $x \in \X$ can be relaxed to that of square integrability 
of the increment $Z(x,\cdot)-Z(y,\cdot)$ for any pair $(x,y) \in \X^2$.

\medskip

The knowledge of the expectation and the non-centered covariance is enough to determine the variance, centered covariance and semivariogram. However, if any function defined on $\X$ with codomain ${\cE}$ can be the expectation of a random field on $\X$, the conditions for a function $\rho$ defined on $\X \times \X$ to be the non-centered covariance function of some random field on $\X$ are largely unknown. A necessary condition \citep{Schoenberg1938} is that $\rho$ must be symmetric and positive semidefinite:
\begin{enumerate}
    \item Symmetry: $\rho(x,y) = \rho(y,x)$ for any $x,y \in \X$.
    \item Positive semidefiniteness: For any positive integer $n$, any set of points $x_1,\ldots,x_n$ in $\X$ and any set of real numbers $\lambda_1,\ldots,\lambda_n$, one has 
        \begin{equation}
        \label{PSD}
            \sum_{k=1}^n \sum_{\ell=1}^n \lambda_{k} \, \lambda_{\ell} \, \rho(x_k,x_\ell) \geq 0.
        \end{equation}
\end{enumerate}

Furthermore, owing to the Daniell-Kolmogorov extension theorem \citep{Billingsley1995}, these conditions are sufficient to ensure the existence of a zero-mean Gaussian random field with covariance function $\rho$, that is, they ensure the compatibility between $\rho$ and a codomain consisting of the entire real line: ${\cE} = \R$. 

Yet, things are much less simple when the codomain is a strict subset of $\R$. A particular case that has been widely examined in the literature is that of the two-point set ${\cE}=\{-1,1\}$, for which conditions on $\rho$ stronger than positive semidefiniteness have been established \citep{McMillan1955, Shepp1967, Matheron1989, Matheron1993, Armstrong1992, Quintanilla2008, Emery2010, Lachieze2015}. The characterization of compatible non-centered covariances for other codomains, such as bounded or half-bounded intervals, is a longstanding problem \citep{Slepian1972} and, to the best of the authors' knowledge, is still unsolved. 
The authors are only aware of the works of \cite{Sondhi1983}, who proposed to generate a random field with a given marginal distribution and a given covariance function by transforming a Gaussian random field,  \cite{Matheron1989}, who examined the compatibility between a covariance model and a given class of positively valued random fields (lognormal random fields in Euclidean spaces), and \cite{Muller2012}, who proposed to generate random fields on the real line valued in $[-1,1]$ with a prescribed stationary covariance function via a spectral simulation method. However, all these works consider specific marginal distributions for the random field, which goes beyond the definition of its codomain.

\medskip

In this context, this article deals with the problem of determining necessary and sufficient conditions that ensure the compatibility between a non-centered covariance function---or other structural tools such as the semivariogram or higher-order spatial moments---and the codomain or set of destination of a random field. We stress that our results apply to real-valued random fields defined on any set of points $\X$; the ambient space containing these points (e.g., a Euclidean space, a sphere, or a graph) is of little importance.

The outline is as follows: Section \ref{background} provides some background material and introduces quantities associated with a codomain and with a real matrix or a real function, which will be referred to under the term \emph{gap}. These quantities will play a key role in the characterization of non-centered covariances, semivariograms, and higher-order moments, as will be presented in Sections \ref{discretesett} (based on matrix gaps) and \ref{continuousett} (based on function gaps). Concluding remarks are given in section \ref{conclusi}. Particular codomains (entire real line, set of relative integers with or without the zero element, two-point sets, bounded intervals, and non-negative half-line) are examined in Appendix \ref{app:particularcases}. The proofs of lemmas and theorems are deferred to Appendix \ref{app:proofs} to ease exposition.


\section{Background material}
\label{background}

Notation: Throughout, an element of $\R^n$ will be represented by a row vector, i.e., a vector whose components are arranged horizontally.

\subsection{Definitions}

\begin{definition}[Trace inner product]
   For any positive integer $n$, the space of real matrices of order $n$ can be endowed with a scalar product called \emph{trace inner product}:
   \begin{equation*}
       \langle \bA,\bB \rangle = \text{tr}(\bA \, \bB^\top) = \sum_{k=1}^n \sum_{\ell=1}^n a_{k \ell} \, b_{k \ell},
   \end{equation*}
   where $\text{tr}(\cdot)$ is the trace operator, $\top$ the transposition, $\bA=[a_{k\ell}]_{k,\ell=1}^n$ and $\bB=[b_{k\ell}]_{k,\ell=1}^n$.\\    
\end{definition}

\begin{definition}[$\gamma$-gap of a real square matrix]
\label{defgap1}
    Let $\cE$ be a subset of $\R$, $n$ be a positive integer, and $\bLambda = [\lambda_{k\ell}]_{k,\ell=1}^n$ be a real square matrix (real-valued two-dimensional array). We define the \emph{$\gamma$-gap} of $\bLambda$ on $\cE$ as
    \begin{equation}
        \label{2Dgap}
        \gamma(\bLambda,{\cE}) = \inf \{ \bz \bLambda \bz^\top: \bz \in {\cE}^n \}.\\
    \end{equation}
\end{definition}

Because the trace is invariant under a cyclic permutation, one can also write:
\begin{equation*}
        \gamma(\bLambda,{\cE}) = \inf \{ \langle \bLambda, \bz^\top \, \bz \rangle: \bz \in {\cE}^n \}.\\
\end{equation*}

The terminology `gap' is borrowed from the concept of gap introduced by \cite{Laurent1996} for a vector (one-dimensional array) of integers $\blambda \in \Z^n$:
\begin{equation*}
   \zeta(\blambda,\{-1,1\}) = \inf \{\lvert \bz \blambda^\top \rvert: \bz \in \{-1,1\}^n \}. 
\end{equation*}

The connection between this vector $\zeta$-gap and our matrix $\gamma$-gap is as follows: if $\blambda \in \Z^n$, $\bLambda = \blambda^\top \blambda$ and $\cE = \{-1,1\}$, then $\gamma(\bLambda,\cE) = \zeta^2(\blambda,\cE)$. \\

\begin{definition}[$\gamma$-gap of a multidimensional array]
\label{defgap3}
    Let $\cE$ be a subset of $\R$, $n$ and $q$ be positive integers, and $\bLambda= [\lambda_{k_1,\ldots,k_q}]_{k_1,\ldots,k_q =1}^n$ be a real-valued $q$-dimensional array. We define the \emph{$\gamma$-gap} of $\bLambda$ on $\cE$ as
    \begin{equation}
        \label{qDgap}
        \gamma(\bLambda,{\cE}) = \inf \left\{ \sum_{k_1=1}^n \ldots \sum_{k_q=1}^n \lambda_{k_1,\ldots,k_q} \, z_{k_1}\ldots z_{k_q} : (z_{1},\ldots, z_{n}) \in {\cE}^n\right\}.
    \end{equation}
For $q=2$, this definition matches Definition \ref{defgap1}.\\
\end{definition}

\begin{definition}[$\eta$-gap of a real square matrix]
\label{defgap2}
    Let $\cE$ be a subset of $\R$, $n$ be a positive integer, and $\bLambda = [\lambda_{k\ell}]_{k,\ell=1}^n$ be a real square matrix (real-valued two-dimensional array). We define the \emph{$\eta$-gap} of $\bLambda$ on $\cE$ as
    \begin{equation}
        \label{etagap}
        \eta(\bLambda,{\cE}) = \sup \left\{ \frac{1}{2} \sum_{k=1}^n \sum_{\ell=1}^n \lambda_{k \ell} \, [z_{k}-  z_{\ell}]^2 : (z_{1}\ldots z_{n}) \in {\cE}^n \right\}.
    \end{equation}
\end{definition}

When no confusion arises, we will simply write `gap' of $\bLambda$ on $\cE$ without specifying whether it is the $\gamma$-gap or the $\eta$-gap.\\

The previous notions of gaps of matrices generalize to that of gaps of functions belonging to a suitable Hilbert space, as per the following definitions.\\

\begin{definition}[Hilbert space of square integrable functions of $\X^2$]
    Let $\mu$ be a positive measure on $\X^2$. We define $L^2(\X^2,\mu)$ as the space of real-valued functions defined on $\X \times \X$ that are square integrable with respect to $\mu$, endowed with the scalar product
\begin{equation*}
    \langle f,g \rangle_{\mu} = \int_{\X^2} f(x,y)\, g(x,y) \, {\rm d}\mu(x,y), \quad f, g \in L^2(\X^2,\mu),
\end{equation*} 
and with the norm $\| f \|_{\mu} = \sqrt{\langle f,f \rangle_{\mu}}$.\\
\end{definition}

\begin{definition}[$\gamma$-gap of a real function]
\label{defgap10}
Let $\cE$ be a subset of $\R$, $\mu$ a finite positive measure on $\X^2$ and $\lambda \in L^2(\X^2,\mu)$. We define the \emph{$\gamma$-gap} of $\lambda$ on $\cE$ as
    \begin{equation}
        \label{generalgap}
        \gamma(\lambda,{\cE},\mu) = \inf \left\{ \int_{\X^2} \lambda(x,y) \, z(x) \, z(y) \, {\rm d}\mu(x,y): z \in {\cE}^\X \text{ and } \| \varphi_z \|_{\mu} < +\infty\right\},
    \end{equation}
where $\varphi_z: (x,y) \mapsto \varphi_z(x,y) = z(x) z(y)$.\\
\end{definition}

\begin{definition}[$\eta$-gap of a real function]
\label{defgap11}
Let $\cE$ be a subset of $\R$, $\mu$ a finite positive measure on $\X^2$ and $\lambda \in L^2(\X^2,\mu)$. We define the \emph{$\eta$-gap} of $\lambda$ on $\cE$ as
    \begin{equation}
        \label{generalgap}
        \eta(\lambda,{\cE},\mu) = \sup \left\{ \frac{1}{2} \int_{\X^2} \lambda(x,y) \, [z(x)-z(y)]^2 \, {\rm d}\mu(x,y): z \in {\cE}^\X \text{ and } \| \psi_z \|_{\mu} < +\infty\right\},
    \end{equation}
where $\psi_z: (x,y) \mapsto \psi_z(x,y) = [z(x)-z(y)]^2$.\\
\end{definition}

\subsection{Properties}

Let $\cE$ be a subset of $\R$ and $\check{\cE}$ be its symmetric with respect to the origin. It is straightforward to establish the following properties:
\begin{itemize}
\item $ 0 \in {\cE} \ \Longrightarrow \ \gamma ( \bLambda, {\cE} ) \leq 0$.
\item $ \gamma ( \bLambda, \check {\cE}) = \gamma (\bLambda, {\cE} )$;
\item ${\cE} \subset {\cE}^\prime \ \Longrightarrow \ \gamma ( \bLambda, {\cE}) \geq \gamma ( \bLambda, {\cE}^\prime )$;
\item $ \gamma ( \bLambda, {\cE} \cup {\cE}^\prime ) \leq \min \left( \gamma ( \bLambda, {\cE} ) , \gamma ( \bLambda, {\cE}^\prime ) \right)$;
\item $ \gamma ( \bLambda, {\cE} \cup \check {\cE} ) \leq \gamma ( \bLambda, {\cE} )$;
\item $ \gamma ( \bLambda, a \, {\cE} ) = a^2 \, \gamma ( \bLambda, {\cE} )$ for $a \in \R$; 
\item $ \gamma (a \,\bLambda, {\cE} ) = a \, \gamma ( \bLambda, {\cE} )$ for $a \in \R_{>0}$;
\item $\gamma(\cdot,\cE)$ is concave:
$$ \gamma \left( \sum_{k=1}^K \omega_k \, \bLambda_k, {\cE} \right) \geq \sum_{k=1}^K \omega_k \, \gamma (\bLambda_k, {\cE} ),$$
for all non-negative real numbers $\omega_1,\ldots,\omega_K$ summing to $1$. This inequality remains valid when $K=+\infty$ and can be extended 
to the continuous case, by replacing the weights $\omega_1,\ldots,\omega_K$ by a probability distribution and the discrete sums by integrals.\\
\end{itemize}

The above properties also hold for the $\gamma$-gap of a function, being a continuous version of the $\gamma$-gap of a matrix, by substituting $\gamma(\lambda,\cE,\mu)$ for $\gamma(\bLambda,\cE)$.\\

As for the $\eta$-gap of a matrix (and, by extension, the $\eta$-gap of a function), one has:
\begin{itemize}
    \item $\eta (\bLambda, {\cE} ) \geq 0$;
    \item $ \eta (\bLambda, {\cE} ) = \eta (\bLambda^\prime, {\cE} )$ if $\bLambda-\bLambda^\prime$ is a diagonal matrix;
    \item $ \eta ( \bLambda, \check {\cE}) = \eta (\bLambda, {\cE} )$;
    \item ${\cE} \subset {\cE}^\prime \ \Longrightarrow \ \eta ( \bLambda, {\cE}) \leq \eta ( \bLambda, {\cE}^\prime )$;
    \item $ \eta ( \bLambda, {\cE} \cup {\cE}^\prime ) \geq \max \left( \eta ( \bLambda, {\cE} ) , \eta ( \bLambda, {\cE}^\prime ) \right)$;
    \item $ \eta ( \bLambda, {\cE} \cup \check {\cE} ) \geq \eta ( \bLambda, {\cE} )$; 
    \item $ \eta ( \bLambda, a \, {\cE} ) = a^2 \, \eta ( \bLambda, {\cE} )$ for $a \in \R$; 
    \item $ \eta (a \,\bLambda, {\cE} ) = a \, \eta ( \bLambda, {\cE} )$ for $a \in \R_{>0}$; 
    \item $\eta(\cdot,\cE)$ is convex: for all non-negative real numbers $\omega_1,\ldots,\omega_K$ summing to $1$, $$\eta \left(\sum_{k=1}^K \omega_k \, \bLambda_k, {\cE} \right) \leq \sum_{k=1}^K \omega_k \, \eta (\bLambda_k, {\cE} ).$$ Again, the inequality remains valid if $K=+\infty$ and can be extended to the continuous case.
\\

\end{itemize}

The $\gamma$-gap and $\eta$-gap of a matrix or of a function can be fully or partially determined for specific families of matrices and/or subsets $\cE$. In particular, the following lemmas hold.\\

\begin{lemma}
\label{lem:negeig0}
    Let $\bLambda$ be a real symmetric positive semidefinite matrix. Then $\gamma(\bLambda,\cE) \geq 0$, and $\gamma(\bLambda,\cE) = 0$ as soon 
    as $0 \in \cE$.
\end{lemma}

\bigskip

\begin{lemma}
\label{lem:negeig}
    Let $\bLambda$ be a real symmetric matrix with at least one negative eigenvalue. Then $\gamma(\bLambda,\cE) = -\infty$ for each of the sets $\R$, $\Q$, $\Z$, and $\Z \smallsetminus \{0\}$. 
\end{lemma}

\bigskip

\begin{corollary}
\label{lem:gammaR}
    Let $\bLambda$ be a real symmetric matrix. Then, $\gamma(\bgL,\R) =\gamma(\bLambda,\Q)= \gamma(\bgL,\Z) = 0$ or $-\infty$.
\end{corollary}

\bigskip

\begin{lemma}
\label{lem:negeig1}
    Let $\bLambda$ be a real symmetric matrix. Then, $\gamma(\bLambda,\R_{\geq 0})=\gamma(\bLambda,\mathbb{N})=0$ or $-\infty$.
\end{lemma}

\bigskip

\begin{lemma}
\label{gamma2eta}
    Let $\bLambda = [\lambda_{k\ell}]_{k,\ell=1}^n$ be a real matrix. Then, $\eta(\bLambda,\cE) = -\gamma(\bLambda-\bDelta,\cE)$,
    where $\bDelta$ is the diagonal matrix whose $k$-th diagonal entry is $\delta_{kk}=\sum_{\ell=1}^n \lambda_{k\ell}$.
\end{lemma}

\bigskip

\begin{corollary}
\label{lem:etaR}
    Let $\bLambda$ be a real symmetric matrix. Then, $\eta(\bgL,\R) =\eta(\bLambda,\Q)=\eta(\bLambda,\R_{\geq 0}) = \eta(\bgL,\Z) = \eta(\bgL,\N) = 0$ or $+\infty$.
\end{corollary}

\bigskip

\begin{corollary}
\label{cor:diagdom}
    If all the off-diagonal entries of $\bLambda$ are non-positive (i.e., $\bLambda$ is a Z-matrix), then $\eta(\bLambda,\cE)=0$.
\end{corollary}

\bigskip

In general, determining the $\gamma$-gap of a given real square matrix $\bLambda$ is an NP-hard problem. This is the case for determining the $\gamma$-gap of a 
matrix on the closed half-line $\R_{\geq 0}$: as it will be shown in the proof of Theorem \ref{matrixgapineq} hereinafter, deciding whether $\gamma(\bLambda,
\R_{\geq 0})=0$ or $\gamma(\bLambda,\R_{\geq 0})=-\infty$ amounts to deciding whether or not $\bLambda$ belongs to the cone of completely positive matrices, 
which is an NP-hard problem \citep{Dickinson2014}. Another example is the computation of the $\gamma$-gap of a symmetric positive semidefinite matrix of rank one, 
which is equivalent to computing the $\zeta$-gap of a real vector, which has been proved to be NP-hard for vectors with entries in $\cE=\{-1,1\}$ \citep{Laurent1996}. Still with $\cE=\{-1,1\}$, determining the $\eta$-gap of a symmetric matrix with non-negative entries is an NP-hard max-cut problem \citep{Goemans1995}.


\section{Gap inequalities in a discrete setting}
\label{discretesett}

In this section, we provide necessary and sufficient conditions for a given mapping to be the non-centered covariance, semivariogram, or higher-order spatial moment, of a random field with codomain $\cE \subseteq \R$. These conditions involve the $\gamma$- and $\eta$-gaps introduced in Definitions \ref{defgap1} to \ref{defgap2}.

\subsection{Characterization of non-centered covariance functions}
\label{covariances}

\begin{theorem}
    \label{matrixgapineq}
    Let ${\cE}$ be a closed subset of the real line.  
    Then, a mapping $\rho: \X \times \X \to \R$ is the non-centered covariance of a random field defined on $\X$ and valued in ${\cE}$ if, and only if, it 
    fulfills the following two conditions:
    \begin{enumerate}
        \item Symmetry: $\rho(x,y)=\rho(y,x)$ for any $x,y \in \X$.
        \item Gap inequalities: for any positive integer $n$, any real square matrix $\bLambda=[\lambda_{k\ell}]_{k,\ell=1}^n$ and any set of points $x_1,
        \ldots,x_n$ in $\X$, one has 
        \begin{equation}
        \label{matrixgap}
            \langle \bLambda, \bR \rangle = \sum_{k=1}^n \sum_{\ell=1}^n \lambda_{k\ell} \, \rho(x_k,x_\ell) \geq \gamma(\bLambda,{\cE}),
        \end{equation}
        where $\bR=[\rho(x_{k},x_{\ell})]_{k,\ell=1}^n$ and $\gamma(\bLambda,{\cE})$ is the $\gamma$-gap of $\bLambda$ on ${\cE}$ as per Definition \ref{defgap1}.
    \end{enumerate}
    Furthermore, the claim of the Theorem holds true if one restricts $\bLambda$ to be a symmetric matrix.
\end{theorem}

\bigskip

\begin{remark}
    The gap inequalities (\ref{matrixgap}) are equivalent to
    \begin{equation*}
            \langle \bLambda, \bR \rangle \leq \sup \left\{ \bz \bLambda \bz^\top: \bz \in \cE^n\right\},
        \end{equation*}
        which can be seen by substituting $\bLambda$ in (\ref{matrixgap}) with $-\bLambda$.
\end{remark}

\bigskip

\begin{remark}
  Theorem \ref{matrixgapineq} generalizes two well-known results (details in Appendix \ref{app:particularcases}):
  \begin{itemize}
      \item For $\cE = \R$, it amounts to stating that a non-centered covariance is a symmetric positive semidefinite function \citep{Schoenberg1938}.
      \item For $\cE = \{-1,1\}$, it amounts to stating that a non-centered covariance is a symmetric corner-positive semidefinite function \citep{McMillan1955}.     
  \end{itemize}

\end{remark}

\bigskip

\begin{remark}
    Theorem \ref{matrixgapineq} may not hold when ${\cE}$ is not closed. To see it, let us consider the case ${\cE} = \R_{> 0}$ (open half-line). In such a case, the gap of any real square matrix $\bLambda$ is negative or zero, insofar as, for any fixed vector $\bz \in \R_{>0}^n$, 
    \begin{equation*}
        \gamma(\bLambda,\cE) \leq \lim_{a \to 0^+} (a\bz) \bLambda (a\bz)^\top = 0.
    \end{equation*}
    
    \bigskip
    
    Accordingly, the symmetry and gap inequality conditions of Theorem \ref{matrixgapineq} are satisfied when $\rho$ is identically zero. However, the zero function cannot be the non-centered covariance of a strictly positive random field on $\X$; it can only be the non-centered covariance of a random field that is almost surely equal to zero at any point of $\X$. The same situation can happen even if ${\cE}$ is bounded, for instance, when it is the open interval $(0,1)$.\\    
\end{remark}

\bigskip

For bounded non-closed sets, one has the following result.\\

\begin{theorem}
    \label{matrixgapineq_boundedset}
    Let ${\cE}$ be a bounded subset of $\R$.  
    A mapping $\rho: \X \times \X \to \R$ is the pointwise limit of a sequence of non-centered covariances of random fields on $\X$ valued in ${\cE}$ if, and only if, it fulfills the following conditions:
    \begin{enumerate}
        \item Symmetry: $\rho(x,y)=\rho(y,x)$ for any $x,y \in \X$.
        \item Gap inequalities: for any positive integer $n$, any real symmetric matrix $\bLambda=[\lambda_{k\ell}]_{k,\ell=1}^n$ and any set of points $x_1,\ldots,x_n$ in $\X$, one has 
        \begin{equation}
        \label{matrixgap_boundedset}
            \langle \bLambda, \bR \rangle \geq \gamma(\bLambda,{\cE}),
        \end{equation}
        where $\bR=[\rho(x_{k},x_{\ell})]_{k,\ell=1}^n$.
    \end{enumerate}
\end{theorem}

\subsection{Characterization of semivariograms}
\label{variograms}

In this section, we restrict ourselves to random fields on $\X$ with no drift, i.e., random fields whose increments have a zero expectation \citep{chiles_delfiner_2012}. In such a case, the semivariogram of a random field $Z = \{Z(x,\omega): x \in \X, \omega \in \Omega \}$ is defined as
\begin{equation*}
   g(x,y) := \frac{1}{2} \int_{\Omega} [Z(x,\omega)-Z(y,\omega)]^2 \, \P({\rm d}\omega), \quad x, y \in \X.
\end{equation*}

\begin{theorem}
    \label{variogr}
    Let ${\cE}$ be a closed subset of $\R$. A mapping $g: \X \times \X \to \R$ is the semivariogram of a random field on $\X$ with no drift and with values in ${\cE}$ if, and only if, it fulfills the following conditions:
    \begin{enumerate}
        \item Symmetry: $g(x,y)=g(y,x)$ for any $x, y \in \X$.
        \item Gap inequalities: for any positive integer $n$, any set of points $x_1,\ldots,x_n$ in $\X$ and any real symmetric matrix $\bLambda= [\lambda_{k \ell}]_{k,\ell=1}^n$, one has 
        \begin{equation*}
            \langle \bLambda,\bG \rangle =\sum_{k=1}^n \sum_{\ell=1}^n \lambda_{k \ell} \, g(x_{k},x_{\ell}) \leq \eta(\bLambda,{\cE}),
        \end{equation*}
        where $\bG=[{g}(x_k,x_{\ell})]_{k,\ell=1}^n$, and $\eta(\bLambda,{\cE})$ is the $\eta$-gap of $\bLambda$ on $\cE$ as per Definition \ref{defgap3}.\\
    \end{enumerate}
\end{theorem}

By choosing diagonal matrices for $\bLambda$, it is seen that the mapping $g$ must be zero on the diagonal of $\X \times \X$. Theorem \ref{variogr} can therefore be restated as follows:\\

\begin{theorem}
    \label{variogr2}
    Let ${\cE}$ be a closed subset of $\R$. A mapping $g: \X \times \X \to \R$ is the semivariogram of a random field on $\X$ with no drift and with values in ${\cE}$ if, and only if, it fulfills the following conditions:
    \begin{enumerate}
        \item Symmetry: $g(x,y)=g(y,x)$ for any $x, y \in \X$.
        \item Value on the diagonal of $\X \times \X$: $g(x,x)=0$ for any $x \in \X$.
        \item Gap inequalities: for any positive integer $n$, any set of points $x_1,\ldots,x_n$ in $\X$ and any real symmetric matrix $\bLambda= [\lambda_{k \ell}]_{k,\ell=1}^n$ with zero diagonal entries, one has 
        \begin{equation}
        \label{gap4variog}
            \langle \bLambda,\bG \rangle =\sum_{k=1}^n \sum_{\ell=1}^n \lambda_{k \ell} \, g(x_{k},x_{\ell}) \leq \eta(\bLambda,{\cE}),
        \end{equation}
        where $\bG=[{g}(x_k,x_{\ell})]_{k,\ell=1}^n$.\\
    \end{enumerate}
\end{theorem}

\begin{remark}
  For $\cE = \R$, Theorem \ref{variogr2} leads to the well-known result that a symmetric function is the variogram of a random field with no drift if, and only if, it vanishes on the diagonal of $\X \times \X$ and is conditionally negative semidefinite, see Appendix \ref{app:particularcases} for details.
\end{remark}


\subsection{Characterization of spatial moments beyond covariance functions}

Theorem \ref{matrixgapineq} can be adapted, mutatis mutandis, to  determine whether a mapping $\rho$ defined on $\X^q$ can be the spatial moment of order $q \in \N_{\geq 2}$ of a random field $Z$ on $\X$ with values in a compact subset of $\R$, i.e., whether one can write $\rho(x_1,\ldots,x_q) = \E(Z(x_1) \ldots Z(x_q))$ for any set of points $x_1,\ldots,x_q$ in $\X$. The proof is a direct extension to that of Theorem \ref{matrixgapineq} (see Appendix \ref{app:proofs}) and is omitted.\\

\begin{theorem}
    \label{highorder}
    Let ${\cE}$ be a closed subset of $\R$ and $q$ an integer greater than $1$. A mapping $\rho_q: \X^q \to \R$ is the $q$-th spatial moment of a random field on $\X$ with values in ${\cE}$ if, and only if, it fulfills the following conditions:
    \begin{enumerate}
        \item Symmetry: $\rho_q(x_1,x_2,\ldots,x_q)=\rho_q(x_{\sigma_1},x_{\sigma_2},\ldots,x_{\sigma_q})$ for any set of points $x_1,\ldots,x_q \in \X$ and any permutation $\{\sigma_1,\ldots,\sigma_q\}$ of $\{1,\ldots,q\}$.
        \item Gap inequalities: for any positive integer $n$, any set of points $x_1,\ldots,x_n$ in $\X$ and any real-valued $q$-dimensional array $\bLambda= [\lambda_{k_1,\ldots,k_q}]_{k_1,\ldots,k_q =1}^n$, one has 
        \begin{equation*}
            \sum_{k_1=1}^n \ldots \sum_{k_q=1}^n \lambda_{k_1,\ldots,k_q} \, \rho_q(x_{k_1},\ldots,x_{k_q}) \geq \gamma(\bLambda,{\cE}),
        \end{equation*}
        where $\gamma(\bLambda,{\cE})$ is the $\gamma$-gap of $\bLambda$ on $\cE$ as per Definition \ref{defgap3}.
    \end{enumerate}
\end{theorem}

\bigskip

\begin{example}
\label{example2}
    Let $q = 2q^\prime$ be an even integer. Then, the mapping
    \begin{equation*}
        \rho_q(x_{1},\ldots,x_{q}) = \text{haf}\big([\rho(x_{k},x_{\ell})]_{k,\ell=1}^q\big),
    \end{equation*}
    where $\rho: \X \times \X \to \R$ is a symmetric positive semidefinite function and $\text{haf}$ is the hafnian, is a valid $q$-th spatial moment of a random field valued in $\R$. Actually, $\rho_q$ is nothing but the $q$-th spatial moment of a zero-mean Gaussian random field with covariance function $\rho$ \citep{Isserlis1918}. 
    
    Recall that the hafnian of a symmetric matrix $\bR=[r_{k\ell}]_{k,\ell=1}^q$ is defined as:
    \begin{equation*}
        \text{haf}(\bR) = \sum_{K,L}  r_{k_1,\ell_1} \ldots r_{k_{q^\prime},\ell_{q^\prime}},
    \end{equation*}
    where the sum is extended over all the decompositions of the set $\{1,\ldots,q\}$ into two disjoint subsets $K=\{k_1,\ldots,k_{q^\prime}\}$ and $L=\{\ell_1,\ldots,\ell_{q^\prime}\}$ such that $k_1<\ldots<k_{q^\prime}$, $\ell_1<\ldots<\ell_{q^\prime}$ and $k_i<\ell_i$ for $i = 1,\ldots,q^\prime$.
\end{example}


\subsection{Multivariate random fields}

For a multivariate random field $\bZ=(Z_1,\ldots,Z_p)$ on $\X$, the non-centered covariance and the semivariogram become matrix-valued:
\begin{equation*}
    \begin{split}
        \brho(x,y) &:= \int_{\Omega} \bZ^\top(x,\omega) \bZ(y,\omega) \, \P({\rm d}\omega), \quad x, y \in \X,\\
        \boldsymbol{g}(x,y) &:= \frac{1}{2} \int_{\Omega} [\boldsymbol{1}^\top \bZ(x,\omega)-\bZ^\top(y,\omega) \boldsymbol{1}]^2 \, \P({\rm d}\omega), \quad x, y \in \X,
    \end{split}
\end{equation*}
the latter being known as the pseudo semivariogram \citep{Myers1991}.\\

All the results of Sections \ref{covariances} and \ref{variograms} generalize to the multivariate setting, by viewing $\bZ$ as a univariate random field defined on $\X \times \{1,\ldots,p\}$. An alternative is to adapt the proofs of the previous theorems to allow the codomains of the univariate random fields $Z_1,\ldots, Z_p$ to be different, i.e., to deal with a codomain of the form ${\cE}={\cE}_1 \times \ldots \times {\cE}_p$ for the $p$-variate random field $\bZ$. In the most general setting, this codomain can be a closed subset of $\R^p$ and not only a Cartesian product of closed sets of $\R$. For instance, in mineral resource evaluation, one can think of jointly modeling an ore grade and a rock type domain by a bivariate random field with codomain ${\cE}=[0,100] \times \{0,1\}$, or modeling a set of compositional variables by a $p$-variate random field with $\cE$ being the $p$-dimensional standard simplex.

This leads to the following straightforward multivariate extensions of Theorems \ref{matrixgapineq} and \ref{variogr2}, which involve multivariate extensions of the $\gamma$-gap and $\eta$-gap stated in Definitions \ref{defgap1} and \ref{defgap2}. The proofs are omitted.\\

\begin{theorem}
    \label{matrixgapineqmultiv}
    Let ${\cE}$ be a closed subset of $\R^p$. A matrix-valued mapping $\brho: \X \times \X \to \R^{p \times p}$ is the non-centered matrix-valued covariance of a $p$-variate random field on $\X$ with values in ${\cE}$ if, and only if, it fulfills the following conditions:
    \begin{enumerate}
        \item Symmetry: $\brho(x,y)=\brho^\top(y,x)$ for any $x,y \in \X$.
        \item Gap inequalities: for any positive integer $n$, any real symmetric matrix $\bLambda = [\lambda_{k \ell}]_{k,\ell=1}^{p\times n}$ and any set of points $x_1,\ldots,x_n$ in $\X$, one has 
        \begin{equation*}
        \label{matrixgapmult}
            \langle \bLambda, \bR \rangle \geq \gamma(\bLambda,{\cE}),
        \end{equation*}
        where $\bR=[\brho(x_k,x_{\ell})]_{k,\ell=1}^n$ and $\gamma(\bLambda,{\cE}) = \inf \{ \bz \bLambda \bz^\top: \bz \in {\cE}^n \}$.\\
    \end{enumerate}
\end{theorem}

In particular, if ${\cE}=\R^p$, the gap inequalities reduce to the positive semidefiniteness restriction (the matrix $\bR$ must be positive semidefinite). \\

\begin{theorem}
    \label{variogr2mult}
    Let ${\cE}$ be a closed subset of $\R^p$. A matrix-valued mapping $\boldsymbol{g}: \X \times \X \to \R^{p \times p}$ is the matrix-valued pseudo semivariogram of a $p$-variate random field on $\X$ with values in ${\cE}$ if, and only if, it fulfills the following conditions:
    \begin{enumerate}
        \item Symmetry: $\boldsymbol{g}(x,y)=\boldsymbol{g}^\top(y,x)$ for any $x, y \in \X$.
        \item Diagonal values: the diagonal entries of $\boldsymbol{g}(x,x)$ are equal to $0$ for any $x \in \X$.
        \item Gap inequalities: for any positive integer $n$, any set of points $x_1,\ldots,x_n$ in $\X$ and any real symmetric matrix $\bLambda = [\lambda_{k \ell}]_{k,\ell=1}^{p\times n}$ with zero diagonal entries, one has 
        \begin{equation*}
           \langle \bLambda, \bG \rangle \leq \eta(\bLambda,{\cE}),
        \end{equation*}
        where $\bG=[\boldsymbol{g}(x_k,x_{\ell})]_{k,\ell=1}^n$ and $\eta(\bLambda,{\cE}) = \sup_{\bz \in {\cE}^n} \{ \bz  (\bDelta-\bLambda) \bz^\top \}$, $\bDelta$ being the diagonal matrix of order $p \times n$ whose $k$-th diagonal entry is the sum of the entries in the $k$-th row of $\bLambda$.\\
    \end{enumerate}
\end{theorem}

In particular, if ${\cE}=\R^p$, the gap inequalities reduce to the conditional negative semidefiniteness restriction of \cite{Gesztesy2017}: the matrix $\bG$ must be conditionally negative semidefinite, i.e., $\blambda \,\bG \, \blambda^\top \leq 0$ for all $\blambda \in \R^{p \times n}$ whose elements sum to zero.\\


\section{Gap inequalities in a continuous setting}
\label{continuousett}

In this section, we propose rewriting the previous theorems in terms of \emph{kernels} rather than matrices, i.e., we trade the discrete framework to a continuous one. \\

\begin{theorem}[Non-centered covariances, compact codomain]
    \label{functiongapineq}
    Let ${\cE}$ be a compact subset of $\R$ and $\mu$ an arbitrary positive finite measure on $\X^2$. A function $\rho \in L^2(\X^2,\mu)$ is the non-centered covariance of a random field on $\X$ with values in ${\cE}$ if, and only if, it fulfills the following conditions:
    \begin{enumerate}
        \item Symmetry: $\rho(x,y)=\rho(y,x)$ for any $x,y \in \X$.
        \item Gap inequalities: for any function $\lambda \in L^2(\X^2,\mu)$, one has 
        \begin{equation}
        \label{functiongap}
            \langle \lambda, \rho \rangle_{\mu}=\int_{\X^2} \lambda(x,y) \, \rho(x,y) \, {\rm d}\mu(x,y) \geq \gamma(\lambda,{\cE},\mu),
        \end{equation}
        where $\gamma(\lambda,{\cE},\mu)$ is the $\gamma$-gap of $\lambda$ on ${\cE}$ as per Definition \ref{defgap10}.
    \end{enumerate}
    Furthermore, the claim of the Theorem holds true if one restricts $\lambda$ and $\mu$ to be symmetric.
\end{theorem}

\bigskip

\begin{remark}
    The gap inequality (\ref{functiongap}) boils down to inequality (\ref{matrixgap}) when $\mu(x,y)$ is the product of the two Dirac measures $\delta_{x}(x_1,\ldots,x_n)$ and $\delta_{y}(x_1,\ldots,x_n)$. However, the counterpart of choosing Dirac measures (which vanish on $\X \smallsetminus \{x_1,\ldots,x_n\}$ and are therefore not positive, but only non-negative) is the need to state the gap inequalities not only for any real square matrix $\bLambda$, but also for any choice of the matrix size ($n$) and of the supporting points $x_1,\ldots,x_n$ in $\X$.    
    In this respect, an advantage of Theorem \ref{functiongapineq} is to replace the discrete formulation of Theorem \ref{matrixgapineq} involving all possible integers $n$ and points $x_1,\ldots,x_n$ in $\X$ by a formulation involving a single positive measure $\mu$ on $\X^2$.\\    
\end{remark}

\bigskip

Mutatis mutandis, Theorem \ref{functiongapineq} can be extended to semivariograms, higher-order spatial moments, and to the multivariate setting, as follows (the proofs are omitted).\\

\begin{theorem}[Semivariograms, compact codomain]
    \label{variogrfunc}
    Let ${\cE}$ be a compact subset of $\R$ and $\mu$ an arbitrary positive finite measure on $\X^2$. A mapping $g \in L^2(\X^2,\mu)$ is the semivariogram of a random field on $\X$ with no drift and with values in ${\cE}$ if, and only if, it fulfills the following conditions:
    \begin{enumerate}
        \item Symmetry: $g(x,y)=g(y,x)$ for any $x, y \in \X$.
        \item Gap inequalities: for any function $\lambda \in L^2(\X^2,\mu)$, one has 
        \begin{equation}
        \label{functiongapvariog}
            \langle \lambda, g \rangle_{\mu}=\int_{\X^2} \lambda(x,y) \, g(x,y) \, {\rm d}\mu(x,y) \leq \eta(\lambda,{\cE},\mu),
        \end{equation}
        where $\eta(\lambda,{\cE},\mu)$ is the $\eta$-gap of $\lambda$ on $\cE$ as per Definition \ref{defgap11}.
    \end{enumerate}
    The function $\lambda$ can be restricted to be zero on the diagonal of $\X \times \X$ provided that the extra condition that the same restriction holds for $g$.\\
\end{theorem}

\begin{theorem}[High-order spatial moments, compact codomain]
    \label{functiongapineq2}
    Let ${\cE}$ be a compact subset of $\R$, $q$ an integer greater than $1$, and $\mu$ an arbitrary positive finite measure on $\X^q$. A function $\rho_q \in L^2(\X^q,\mu)$ is the $q$-th spatial moment of a random field on $\X$ with values in ${\cE}$ if, and only if, it fulfills the following conditions:
    \begin{enumerate}
        \item Symmetry: $\rho_q(x_1,x_2,\ldots,x_q)=\rho_q(x_{\sigma_1},x_{\sigma_2},\ldots,x_{\sigma_q})$ for any set of points $x_1, \ldots, x_q \in \X$ and any permutation $\{\sigma_1, \ldots, \sigma_q\}$ of $\{1,\ldots,q\}$.
        \item Gap inequalities: for any function $\lambda \in L^2(\X^q,\mu)$, one has 
        \begin{equation}
        \label{functiongap2}
            \int_{\X^q} \lambda(\bx) \, \rho_q(\bx) \, {\rm d}\mu(\bx) \geq \gamma(\lambda,{\cE},\mu),
        \end{equation}
        where $\bx = (x_1,\ldots,x_q)$ and 
        \begin{equation}
        \label{generalgap2}
        \gamma(\lambda,{\cE},\mu) = \inf \left\{ \int_{\X^q} \lambda(\bx) \, \varphi_z(\bx) \, {\rm d}\mu(\bx): z \in {\cE}^\X \text{ and } \varphi_z \in L^2(\X^q,\mu)\right\}.
    \end{equation}
    with $\varphi_z(\bx) = z(x_1) \times \ldots \times z(x_q)$.\\
    \end{enumerate}
\end{theorem}

\begin{theorem}[Matrix-valued covariances, compact codomain]
    \label{functiongapineq3}
    Let ${\cE}$ be a compact subset of $\R^p$ and $\mu$ an arbitrary positive finite measure on $\X^2$. A matrix-valued function $\brho = [\rho_{ij}]_{i,j=1}^p$ with $\rho_{ij} \in L^2(\X^2,\mu)$ is the non-centered covariance of a $p$-variate random field on $\X$ with values in ${\cE}$ if, and only if, it fulfills the following conditions:
    \begin{enumerate}
        \item Symmetry: $\brho(x,y)=\brho^\top(y,x)$ for any $x,y \in \X$.
        \item Gap inequalities: for any matrix-valued function $\blambda = [\lambda_{ij}]_{i,j=1}^p$ with $\lambda_{ij} \in L^2(\X^2,\mu)$, one has 
        \begin{equation}
        \label{functiongap3}
            \langle \blambda, \brho \rangle_{\mu} := \int_{\X^2} \text{tr}\left(\blambda(x,y) \, \brho(y,x)\right) \, {\rm d}\mu(x,y) \geq \gamma(\blambda,{\cE},\mu),
        \end{equation}
        where $\gamma(\blambda,{\cE},\mu) = \inf \left\{ \int_{\X^2} \text{tr}\left(\bz(x) \, \blambda(x,y) \, \bz^\top(y) \right) {\rm d}\mu(x,y): \bz \in {\cE}^\X \right\}$.\\
    \end{enumerate}
\end{theorem}

The multivariate case can also be dealt with by viewing a $p$-variate random field on $\X$ as a univariate random field on $\X \times \{1,\ldots,p\}$, which amounts to replacing $\X$ by $\X \times \{1,\ldots,p\}$ in Theorem \ref{functiongapineq}. This alternative assumes that all the field components are valued in the same compact $\cE$, hence it is less flexible than Theorem \ref{functiongapineq3}.

\bigskip

The case of non-compact codomains is more complicated to deal with. A clean treatment needs additional assumptions on the set $\X$ (to be a metric space), on the class of admissible covariance functions (to be continuous), and on the measure $\mu$ (to be a product measure), as indicated in the following theorem for the case when $\cE = \R$, $\R_{\geq 0}$ or $\R_{\leq 0}$.\\

\begin{theorem}[Non-centered covariances, unbounded codomains]
    \label{functiongapineqRplus}
    Let $\X$ be a metric space, $\varpi$ a positive finite measure on $\X$, $\mu = \varpi \times \varpi$ the corresponding product measure on $\X^2$, and $\cE = \R$, $\R_{\geq 0}$ or $\R_{\leq 0}$. A continuous function $\rho: \X \times \X \to \R$ is the non-centered covariance of a random field on $\X$ with values in $\cE$ if, and only if, it fulfills the following conditions:
    \begin{enumerate}
        \item Symmetry: $\rho(x,y)=\rho(y,x)$ for any $x,y \in \X$.
        \item Gap inequalities: for any symmetric continuous function $\lambda: \X \times \X \to \R$, one has 
        \begin{equation}
        \label{functiongapRplus}
            \langle \lambda,\rho\rangle_{\mu}=\int_{\X^2} \lambda(x,y) \, \rho(x,y) \, {\rm d}\mu(x,y) \geq \gamma(\lambda,\cE,\mu).
        \end{equation}
    \end{enumerate}
\end{theorem}

\begin{remark}
    If $\cE=\R$, $\varpi$ is a measure with a continuous density and $\lambda$ is a separable function, the gap inequalities (\ref{functiongapRplus}) boil down to Mercer's condition defining functions of positive type (aka positive semidefinite kernels) on $\X \times \X$:
    \begin{equation}
        \label{Mercer}
            \int_{\X^2} g(x) \, \rho(x,y) \, g(y) \, {\rm d}x \, {\rm d}y \geq 0,
        \end{equation}
    for any $g$ that is continuous and square integrable on $\X$ \citep{Mercer1909}.\\
    
\end{remark}

\section{Concluding remarks}
\label{conclusi}

We have derived a set of inequalities that are necessary and sufficient for a symmetric function to be the non-centered covariance, semivariogram, or higher-order moment, of a random field with index set $\X$ and codomain $\cE$ that is a closed or a compact subset of $\R$. Such inequalities generalize known results, in particular, the fact that the class of non-centered covariances coincides with the class of symmetric positive semidefinite functions when $\cE = \R$, and symmetric corner-positive semidefinite functions when $\cE = \{-1,1\}$, while the class of semivariograms coincides with the class of symmetric conditionally negative semidefinite functions that vanish on the diagonal of $\X \times \X$ when $\cE = \R$. In the continuous framework, one also retrieves Mercer's condition on positive semidefinite operators.

The key components of each inequality are
\begin{enumerate}
    \item a test matrix $\bLambda$ (discrete framework) or a test function $\lambda$ (continuous framework) that plays the role of the lens through which a tentative covariance, semivariogram or higher-order moment is investigated;
    \item a quantity that we named \emph{gap} that depends on the codomain and on the test matrix or test function, but not on the index set $\X$ nor on the tentative covariance, semivariogram or higher-order moment under consideration.\\
\end{enumerate}

The presented formalism shows connections not only with the theory of probability and stochastic processes, but also with topology, algebra, analysis, combinatorial optimization, and convex geometry. Our results give an insight into the spectral theory of covariance kernels that are realizable, given a codomain, a theory that is still in its infancy.

\section*{\textbf{Acknowledgments}}
This work was funded and supported by the National Agency for Research and Development of Chile [grants ANID CIA250010 and ANID Fondecyt 1250008].

\section*{\textbf{Declarations}}

\bmhead{Conflict of Interest} The authors declare no knowledge of meeting financial interests or personal relationships that could have appeared to influence the work reported in this paper. This article does not contain any studies involving human participants performed by the authors.

\appendix

\section{A look at particular codomains}
\label{app:particularcases}

\subsection{Random fields valued in $\R$ or $\Z$}

\begin{theorem}
\label{th:PSD}
Let $\cE=\R$ or $\Z$. A mapping $\rho: \X \times \X \to \R$ is the non-centered covariance of a random field on $\X$ with values in $\cE$ if, and only if, the following conditions hold:
    \begin{enumerate}
        \item Symmetry: $\rho(x,y)=\rho(y,x)$ for any $x,y \in \X$.
        \item Positive semidefiniteness: For any positive integer $n$, any set of points $x_1,\ldots,x_n$ in $\X$ and any set of real numbers $\lambda_1,\ldots,\lambda_n$, the inequality (\ref{PSD}) holds.\\
    \end{enumerate}
\end{theorem}

\begin{theorem}
    \label{variogr_unboundedset}
    Let $\cE=\R$ or $\Z$. A mapping $g: \X \times \X \to \R$ is the semivariogram of a random field on $\X$ with values in $\cE$ if, and only if, it fulfills the following conditions:
    \begin{enumerate}
        \item Symmetry: $g(x,y)=g(y,x)$ for any $x, y \in \X$.
        \item Value on the diagonal of $\X \times \X$: $g(x,x)=0$ for any $x \in \X$.
        \item Conditional negative semidefiniteness: for any positive integer $n$, any set of points $x_1,\ldots,x_n$ in $\X$ and any set of real numbers $\lambda_1,\ldots,\lambda_n$ that sum to zero, one has 
        \begin{equation}
        \label{condneg}
            \sum_{k=1}^n \sum_{\ell=1}^n \lambda_{k} \,\lambda_{\ell} \, g(x_{k},x_{\ell}) \leq 0.\\
        \end{equation}
    \end{enumerate}
\end{theorem}

\subsection{Random fields valued in $\Z \smallsetminus \{0\}$}

\begin{lemma}
\label{lem:Nplus}
    For a real symmetric positive semidefinite matrix $\bLambda$, $\gamma(\bLambda,\Z\smallsetminus\{0\})= \gamma_n  \det(\bLambda)^{\frac{1}{n}}$, where $\gamma_1$, $\gamma_2$, $\ldots$, are the so-called Hermite's constants. In particular, one has \citep{Blichfeldt1929, Cassels1997, Conway1998},
    \begin{itemize}
        \item $[\gamma_1, \gamma_2,\gamma_3,\gamma_4,\gamma_5, \gamma_6, \gamma_7, \gamma_8, \gamma_{24}] = [1, \sqrt{\frac{4}{3}}, \sqrt[3]{2}, \sqrt{2}, \sqrt[5]{8}, \sqrt[6]{\frac{64}{3}}, \sqrt[7]{64},2,4]$
        \item $\frac{1}{\pi} \left[2\Gamma \left(1+\frac{n}{2}\right)\right]^{\frac{2}{n}} \leq \gamma_n \leq \frac{2}{\pi} \, \Gamma \left(2+\frac{n}{2}\right)^{\frac{2}{n}}$
        \item $\gamma_n \leq n$ for any $n \in \N$
        \item $\gamma_n \leq \frac{2n}{3}$ for any $n \in \N \smallsetminus \{0 \}$
        \item $\frac{1}{2\pi \mathsf{e}} \lesssim \frac{\gamma_n}{n} \lesssim \frac{1.744}{2\pi \mathsf{e}}$ for large $n$.
    \end{itemize}
\end{lemma}

\bigskip
\begin{theorem}
\label{th:PSDZ0}
Let $\rho: \X \times \X \to \R$ be a symmetric positive semidefinite function. Then, $\rho + \varepsilon \, \delta$, defined as
\begin{equation*}
    \rho(x,y) + \varepsilon \, \delta(x,y) = \begin{cases}
    \rho(x,y) \text{ if $x \neq y$}\\
    \rho(x,x) + \varepsilon \text{ otherwise,}
    \end{cases}
\end{equation*}
is the non-centered covariance of a random field on $\X$ with values in $\Z \smallsetminus \{0\}$, provided that $\varepsilon \geq 1$. If, furthermore, $\rho(x,x) \geq \frac{1}{3}$ for any $x \in \X$, then the condition on $\varepsilon$ can be reduced to $\varepsilon \geq \frac{2}{3}$. \\
\end{theorem}

\subsection{Binary random fields}

\begin{definition}[\citealp{McMillan1955}]
    A unit process is a random field valued in $\{-1,1\}$.\\
\end{definition}

\begin{definition}[\citealp{McMillan1955}]
\label{def:CP}
    A square matrix $\bLambda$ is corner positive if $\gamma(\bLambda,\{-1,1\}) \geq 0$.\\
\end{definition}

\begin{theorem}[\citealp{McMillan1955}]
\label{th:CP}
A mapping $\rho: \X \times \X \to \R$ is the non-centered covariance of a unit process in $\X$ if, and only if, it fulfills the following conditions:
    \begin{enumerate}
        \item Symmetry: $\rho(x,y)=\rho(y,x)$ for any $x,y \in \X$.
        \item Value on the diagonal of $\X \times \X$: $\rho(x,x)=1$ for any $x \in \X$.
        \item Corner positive inequalities: for any positive integer $n$, any set of points $x_1,\ldots,x_n$ in $\X$ and any corner positive matrix $\bLambda=[\lambda_{k\ell}]_{k,\ell=1}^n$, one has 
        \begin{equation}
        \label{corner}
            \sum_{k=1}^n \sum_{\ell=1}^n \lambda_{k\ell} \, \rho(x_k,x_\ell) \geq 0.
        \end{equation}
    \end{enumerate}
\end{theorem}

\begin{theorem}
\label{thMcMillan}
    Let $\rho: \X \times \X \to [-1,1]$ be the non-centered covariance of a random field valued in $[-1,1]$. Then, the mapping $\rho^*$ defined by
    \begin{equation*}
        \rho^*(x,y)=\begin{cases}
            \rho(x,y) \text{ if $x \neq y$}\\
            1 \text{ otherwise,}
        \end{cases}
    \end{equation*}
    is the non-centered covariance of a unit process in $\X$.
\end{theorem}

\bigskip

The next theorem exhibits a class of real symmetric matrices $\bLambda$ for which one can calculate the gap $\gamma(\bLambda,\{-1,1\})$ without calculating $\bz \bLambda \bz^\top$ for all the possible realizations $\bz \in \{-1,1\}^n$. The application of Theorem \ref{matrixgapineq} to these matrices therefore provides necessary conditions for a given mapping $\rho$ to be a realizable non-centered covariance of a unit process.\\

\begin{theorem}
\label{th:necessary-11}
Necessary conditions for a mapping $\rho: \X \times \X \to \R$ to be the non-centered covariance of a unit process in $\X$ are:
    \begin{enumerate}
        \item Symmetry: $\rho(x,y)=\rho(y,x)$ for any $x,y \in \X$.
        \item Hadamard transform inequalities: for any positive integer $n$, any set of points $x_1,\ldots,x_n$ in $\X$ and any integers $u,v$ in $\{1,\ldots,n\}$, one has 
        \begin{equation}
        \label{Walsh}
            \sum_{k=1}^n \sum_{\ell=1}^n\lambda_{k \ell}(u,v) \, \rho(x_k,x_\ell) \geq b_m(n-q(u,v)) - q(u,v)^2,
        \end{equation}
        where 
        \begin{itemize}
            \item $\lambda_{k\ell}(u,v) = (-1)^{\sum_{j=1}^m [b_j(k) b_j(u)+b_j(\ell) b_j(v)]}$
            \item $\boldsymbol{b}(a)=[b_j(a)]_{j=1}^m$ is the binary representation of $a$
            \item $m = 1+ \lfloor \log_2(n) \rfloor$, with $\lfloor \cdot \rfloor$ the floor function
            \item $q(u,v) = b_m(\boldsymbol{r}(u,v)) \, \boldsymbol{1}^\top$
            \item $\boldsymbol{r}(u,v)=[r_i(u,v)]_{i=1}^n$ with $r_i(u,v)=\boldsymbol{b}(u) \veebar \boldsymbol{b}(v)) \, \boldsymbol{b}(i)^\top$
            \item $b_m(\boldsymbol{a})=[b_m(a_1),\ldots,b_m(a_n)]$, with $b_m(\cdot)$ the rightmost bit of the binary representation (least significant bit)
            \item $\veebar$ is the exclusive OR (bitwise addition modulo $2$)
            \item $\boldsymbol{1}$ is an $n$-dimensional row vector of ones.
        \end{itemize}
    In particular, if $u=v$, then $q(u,v)=0$ and the right-hand side of (\ref{Walsh}) boils down to $b_m(n)$, i.e., $0$ if $n$ is even and $1$ if $n$ is odd.\\
    \end{enumerate} 
\end{theorem}

\begin{theorem}
\label{variogr_unitprocess}
    A mapping $g: \X \times \X \to [0,2]$ is the semivariogram of a unit process in $\X$ if, and only if, it fulfills the following conditions:
    \begin{enumerate}
        \item Symmetry: $g(x,y)=g(y,x)$ for any $x,y \in \X$.
        \item Gap inequalities: for any positive integer $n$, any set of points $x_1,\ldots,x_n$ in $\X$ and any real matrix $\bLambda=[\lambda_{k\ell}]_{k,\ell=1}^n$, one has 
        \begin{equation}
        \label{matrixgap2}
            \sum_{k=1}^n \sum_{\ell=1}^n \lambda_{k\ell} \, g(x_k,x_\ell) \leq \eta(\bLambda,\{-1,1\}) = \sigma(\bLambda) - \gamma(\bLambda,\{-1,1\}),
        \end{equation}
        where $\sigma(\bLambda) = \sum_{k=1}^n \sum_{\ell=1}^n  \lambda_{k\ell}$.  
    \end{enumerate}
\end{theorem}

\bigskip

\begin{remark}
As a particular case of Theorem \ref{variogr_unitprocess}, if $\bLambda = \blambda^\top \, \blambda$ with $\blambda$ an $n$-dimensional row vector with entries $\lambda_1, \ldots,\lambda_n$, Eq. (\ref{matrixgap2}) becomes
\begin{equation}
        \label{matrixgap3}
            \sum_{k=1}^n \sum_{\ell=1}^n \lambda_{k\ell} \, g(x_k,x_\ell) \leq \sigma^2(\blambda) - \zeta^2(\blambda),
\end{equation}
where $\sigma(\blambda) = \sum_{k=1}^n \lambda_k$ and $\zeta(\blambda) = \inf \{ |\bz \blambda^\top|: \bz \in \{-1,1\}^n \}$ is the $\zeta$-gap of vector $\blambda$. Equivalently, the mapping $g/4$, which is the semivariogram of a random field valued in $\{0,1\}$, fulfills the gap inequalities defined by \cite{Laurent1996}. The latter inequalities imply many other well-known inequalities, in particular, the negative-type and hypermetric inequalities \citep{Galli2012}.\\
\end{remark}

\begin{theorem}[\citealp{emery2025}]
A mapping $g: \X \times \X \to \R$ is the semivariogram of a unit process with no drift in $\X$ if, and only if, it has the following representation:
\begin{equation*}
   g(x,y) = \frac{2}{\pi} \int_0^1 \arccos C_t(x,y) \, {\rm d}t,
\end{equation*}
where, for any $t \in [0,1]$, $C_t: \X \times \X \to [-1,1]$ is a symmetric positive semidefinite mapping that is equal to $1$ on the diagonal of $\X \times \X$.\\   

\end{theorem}

\subsection{Random fields valued in a bounded and closed interval}

\begin{theorem}
\label{th:boundedinterval}
Necessary, but not sufficient, conditions for a mapping $\rho: \X \times \X \to \R$ to be the non-centered covariance of a random field on $\X$ with values in $[-1,1]$ are:
    \begin{enumerate}
        \item Symmetry: $\rho(x,y)=\rho(y,x)$ for any $x,y \in \X$.
        \item Boundedness: $\rho(x,y) \in [-1,1]$ for any $x,y \in \X$.
        \item Positive semidefiniteness: for any positive integer $n$, any set of points $x_1,\ldots,x_n$ in $\X$ and any set of real numbers $\lambda_1,\ldots,\lambda_n$, one has 
        \begin{equation}
        \label{PSD3}
            \sum_{k=1}^n \sum_{\ell=1}^n \lambda_{k} \, \lambda_{\ell} \, \rho(x_k,x_\ell) \geq 0.
        \end{equation}
    \end{enumerate}   
\end{theorem}

The statement of Theorem \ref{th:boundedinterval} is somehow perturbing, as it implies that, given a symmetric positive semidefinite function (even a bounded one), there may not exist a bounded random field with this function as its non-centered covariance. An example is given by \cite{McMillan1955}: for $\sigma \in (\frac{2\sqrt{2}}{\pi},1]$ and $\theta \in \R\smallsetminus \mathbb{Q}$, the mapping $\rho: \Z \times \Z \to [-1,1]$ defined by
\begin{equation*}
    \rho(x,y) = \begin{cases}
        \sigma^2\cos(2\pi \theta (x-y)) \text{ if $x \neq y$}\\
        1 \text{ otherwise},
    \end{cases}
\end{equation*}
is symmetric, bounded and positive semidefinite, but is not the covariance function of any random field on $\Z$ valued in $[-1,1]$. Simpler examples are the pure cosine covariance ($\sigma=1$ and $\theta \in \R$) in $\R \times \R$ and, more generally, any correlation function $\rho$ in $\X \times \X$ that does not belong to the set of unit process covariance functions (see previous section): the fact that it is equal to $1$ on the diagonal of $\X \times \X$ implies that a random field valued in $[-1,1]$ having $\rho$ as its covariance would necessarily be a unit process, but such processes do not admit stationary covariance functions that are smooth at the origin \citep{Matheron1989}. \\

A necessary and sufficient condition is given in the next theorem.\\

\begin{theorem}
\label{th:boundedinterval20}
A necessary and sufficient condition for a mapping $\rho: \X \times \X \to \R$ to be the non-centered covariance of a random field on $\X$ with values in $[-1,1]$ is to be of the form
\begin{equation}
\label{sufficientrho10}
    \rho(x,y) = \frac{1}{2\pi} \int_{-1}^1 \int_{-1}^1 \int_{0}^1  \arcsin C_{t}((x,u),(y,v)) {\rm d}t \, {\rm d}u \, {\rm d}v, \quad x, y \in \X,
\end{equation}
where, for any $t \in [0,1]$, $C_t: (\X \times [-1,1])^2\to [-1,1]$ is a symmetric positive semidefinite mapping that is equal to $1$ on the diagonal of $(\X \times [-1,1])^2$.\\   
\end{theorem}

\begin{example}
    Let $a \in [0,1]$ and consider the separable covariance, independent of $t$:
    \begin{equation*}
        C_t((x,u),(y,v))= C(x,y) \, C^\prime(u,v)
    \end{equation*}
    with 
    \begin{equation*}
       C^\prime(u,v) = \begin{cases}
           1 \text{ if $u=v$}\\
           a \text{ otherwise.}
       \end{cases} 
    \end{equation*}
    The representation (\ref{sufficientrho10}) leads to the following mapping:
    \begin{equation*}
        \rho_a(x,y) = \frac{2}{\pi} \arcsin (a C(x,y)), \quad x, y \in \X,
    \end{equation*}
    which is the non-centered covariance of a random field valued in $[-1,1]$. For instance, if $a=\frac{1}{2}$, it is known that $\rho_a$ is the covariance of the $[-1,1]$-uniform transform of a standard Gaussian random field with covariance $C$ \citep{Sondhi1983}, while for $a=1$, it is the non-centered covariance of a unit process \citep{McMillan1955}.\\
\end{example}

\subsection{Random fields valued in $\R_{\geq 0}$}

\begin{definition}
\label{coposit}
    A real symmetric matrix $\bLambda$ is said to be copositive if $\gamma(\bLambda,\R_{\geq 0}) \geq 0$ \citep[Definition 1.1]{Hiriart2010}. \\ 
\end{definition}

\begin{definition}
    A real symmetric matrix $\bR$ is said to be completely positive if it can be factorized as $\bR=\bB \bB^\top$, where $\bB$ is a (not necessarily square) matrix with non-negative entries \citep{Hall1963}. \\ 
\end{definition}

\begin{definition}
    A real symmetric matrix $\bR$ is said to be doubly non-negative if it is positive semidefinite and has nonnegative entries. \\ 
\end{definition}

\begin{theorem}
    \label{complpos}
Necessary and sufficient conditions for a mapping $\rho: \X \times \X \to \R$ to be the non-centered covariance of a random field on $\X$ valued in ${\R_{\geq 0}}$ are:
    \begin{enumerate}
        \item Symmetry: $\rho(x,y)=\rho(y,x)$ for any $x,y \in \X$.
        \item Complete positivity: for any positive integer $n$ and any set of points $x_1,\ldots,x_n$ in $\X$, $\bR =[\rho(x_{k},x_{\ell})]_{k,\ell=1}^n$ is completely positive. Equivalently, for any copositive matrix $\bLambda =[\lambda_{k\ell}]_{k,\ell=1}^n$, one has 
        \begin{equation*}
            \langle \bLambda, \bR\rangle=\sum_{k=1}^n \sum_{\ell=1}^n \lambda_{k\ell} \, \rho(x_k,x_\ell) \geq 0.
        \end{equation*}\\
    \end{enumerate}
    
\end{theorem}

\begin{remark}
    The set of completely positive kernels is a closed convex cone that is an infinite dimensional analog of the cone of completely positive matrices of finite order. For topological descriptions of this cone, the reader is referred to \cite{Dobre2016} and \cite{Burgdorf2017}.\\  
\end{remark}

\begin{theorem}
    \label{Rplus1}
    Necessary conditions for a mapping $\rho: \X \times \X \to \R$ to be the non-centered covariance of a random field on $\X$ valued in ${\R_{\geq 0}}$ are:
    \begin{enumerate}
        \item Symmetry: $\rho(x,y)=\rho(y,x)$ for any $x,y \in \X$.
        \item Non-negativity: $\rho(x,y) \geq 0$ for any $x,y \in \X$.
        \item Positive semidefiniteness: for any positive integer $n$, any set of points $x_1,\ldots,x_n$ in $\X$ and any set of real numbers $\lambda_1,\ldots,\lambda_n$, one has 
        \begin{equation}
        \label{PSD2}
            \sum_{k=1}^n \sum_{\ell=1}^n \lambda_{k} \, \lambda_{\ell} \, \rho(x_k,x_\ell) \geq 0.
        \end{equation}
    \end{enumerate}
    These conditions are sufficient only if $\X$ contains at most four points.\\
\end{theorem}

\begin{theorem}
    \label{R0}
    Sufficient conditions for a mapping $\rho: \X \times \X \to \R$ to be the non-centered covariance of a random field on $\X$ valued in $\R_{\geq 0}$ are:
    \begin{enumerate}
        \item Symmetry: $\rho(x,y)=\rho(y,x)$ for any $x,y \in \X$.
        \item Positivity: $\rho(x,y) > 0$ for any $x,y \in \X$.
        \item Log-positive semidefiniteness: for any positive integer $n$, any set of points $x_1,\ldots,x_n$ in $\X$ and any set of real numbers $\lambda_1,\ldots,\lambda_n$, one has 
        \begin{equation}
        \label{PSD200}
            \sum_{k=1}^n \sum_{\ell=1}^n \lambda_{k} \, \lambda_{\ell} \, \ln(\rho(x_k,x_\ell)) \geq 0.\\
        \end{equation}
    \end{enumerate}
\end{theorem}


\section{Proofs}
\label{app:proofs}

\begin{proof}[Proof of Lemma \ref{lem:negeig0}]
    By definition of positive semidefiniteness, the $\gamma$-gap of $\bLambda$ is non-negative. Moreover, $\bz \bLambda \bz^\top = 0$ when $\bz = \boldsymbol{0}$,
    which concludes the proof.\\
\end{proof}

\begin{proof}[Proof of Lemma \ref{lem:negeig}]
    There exists $\bz_0 \in \R^n$ such that $\bz_0 \, \bgL \, \bz_0^\top <0$. By continuity, $\R^n$ contains a neighborhood ${\cal N}$ of $\bz_0$
    such that $ \bz \, \bgL \, \bz^\top < 0$ for any $ \bz \in {\cal N}$. Within this neighborhood, one can find $ \bz_1$ with rational and non-zero coordinates.
    Accordingly, there exists an integer $a$ such that $ a \bz_1$ belongs to $(\Z\smallsetminus \{0\})^n$ and $ (a \bz_1) \, \bgL \, ( a \bz_1)^\top <0$. Since $a$ can be
    chosen arbitrarily large, $ a^2 \, \bz_1 \bgL \, \bz_1^\top $ can also be arbitrarily large in magnitude, which proves $ \gamma(\bgL,\Z\smallsetminus \{0\}) = -
    \infty$. The result follows since the sets $\R$, $\Q$ and $\Z$ contain $\Z\smallsetminus \{0\}$. \\
\end{proof}

\begin{proof}[Proof of Lemma \ref{lem:negeig1}]
    We first note that $\boldsymbol{z} \, \bLambda \, \boldsymbol{z}^\top =0$ when $\boldsymbol{z}$ is a vector of zeros, so that both $\gamma(\bLambda,\R_{\geq 0})$ and $\gamma(\bLambda,\mathbb{N})$ are non-positive. Let us now distinguish two cases:
    \begin{enumerate}
        \item \textbf{$\gamma(\bLambda,\R_{\geq 0}) = 0$}. Since $\N \subset \R_{\geq 0}$, we have that $\gamma(\bLambda,\R_{\geq 0}) \leq \gamma(\bLambda,\mathbb{N})$; thus, $\gamma(\bLambda,\mathbb{N})=0$.
        \item \textbf{$\gamma(\bLambda,\R_{\geq 0}) < 0$}. Then, $\bLambda$ has at least one negative eigenvalue and the proof of the lemma is similar to that of Lemma \ref{lem:negeig}, by replacing $\R^n$ and $(\Z \smallsetminus \{0\})^n$ by $\R_{\geq 0}^n$ and $\N^n$, respectively.\\
    \end{enumerate}
\end{proof}

\begin{proof}[Proof of Lemma \ref{gamma2eta}]
    The proof relies on the following identity, valid for any $\bz=(z_1,\ldots,z_n)$:
    \begin{equation*}
        \frac{1}{2} \sum_{k=1}^n \sum_{\ell=1}^n \lambda_{k \ell} \, [z_{k}-  z_{\ell}]^2 =\bz (\bDelta-\bLambda) \bz^\top.
    \end{equation*}
\end{proof}

\begin{proof}[Proof of Corollary \ref{cor:diagdom}]
   Under the stated conditions, $\bLambda-\bDelta$ is a diagonally dominant matrix with non-negative diagonal entries, hence it is positive semidefinite and its $\gamma$-gap is positive or zero (Lemma \ref{gamma2eta}). For a $n$-dimensional vector $\bz$ whose entries are all equal to the same element $z$ of $\cE$, it is seen that $\bz (\bLambda-\bDelta) \bz^\top = 0$, which concludes the proof.  \\
\end{proof}

\begin{proof}[Proof of Theorem \ref{matrixgapineq}]
\noindent \textit{Necessity.}
Let $Z = \{Z(x,\omega): x\in \X, \omega \in \Omega \}$ be a random field taking values in ${\cE}$. We have
$$ \sum_{k=1}^n \sum_{\ell=1}^n \lambda_{k\ell} \, Z(x_k,\omega) \, Z(x_\ell,\omega) \geq \inf \left\{\sum_{k=1}^n \sum_{\ell=1}^n \lambda_{k,\ell} \, z_k \, z_\ell : (z_1,\ldots,z_n) \in {\cE}^n \right\}, \quad \omega \in \Omega. $$
By definition of the gap, this gives
$$ \sum_{k=1}^n \sum_{\ell=1}^n \lambda_{k\ell} \, Z(x_k,\omega) \, Z(x_\ell,\omega) \geq \gamma(\bLambda,{\cE}),  \quad \omega \in \Omega. $$
It then remains to take the expectation of both sides to obtain (\ref{matrixgap}).

\medskip

\noindent \textit{Sufficiency.}
We first prove that the sufficiency conditions can be restricted to real symmetric matrices $\bLambda$. On the one hand, any real square matrix $\bLambda$ is the sum of a symmetric matrix ${\bSigma}=[\sigma_{k\ell}]_{k,\ell=1}^n$ and an antisymmetric matrix ${\boldsymbol{A}}=[\alpha_{k\ell}]_{k,\ell=1}^n$. On the other hand, for any $z_1,\ldots,z_n \in {\cE}$ and $x_1,\ldots,x_n \in \X$, $\sum_{k=1}^n \sum_{\ell=1}^n \alpha_{k\ell} \, z_k \, z_\ell=\sum_{k=1}^n \sum_{\ell=1}^n \alpha_{k\ell} \, \rho(x_k,x_\ell)=0$ due to the symmetry of $\rho$. In particular, this implies $\gamma(\bLambda,{\cE})=\gamma({\bSigma},{\cE})$. Accordingly, the gap inequalities (\ref{matrixgap}) are equivalent to
\begin{equation*}
    \sum_{k=1}^n \sum_{\ell=1}^n \sigma_{k\ell} \, \rho(x_k,x_\ell) \geq \gamma({\bSigma},{\cE}).
\end{equation*}

\noindent To close the proof of the sufficiency part, we distinguish four cases, depending on whether ${\cE}$ is the real line, the closed half-line $\R_{\geq 0}$ or $\R_{\leq 0}$, a compact subset, or a closed subset; the latter is the most general case, but we provide proofs for the former three cases that are of independent interest.\\

\noindent \textbf{Case 1: ${\cE}=\R$}. Let $\bLambda$ be a real symmetric matrix. If it has at least one negative eigenvalue, then the gap inequalities (\ref{matrixgap}) are automatically fulfilled, on account of Lemma \ref{lem:negeig}. If all the eigenvalues of $\bLambda$ are non-negative, i.e., if $\bLambda$ is positive semidefinite, then $\gamma(\bLambda,\R)=0$ (Lemma \ref{lem:negeig0}) and the gap inequalities become
    \begin{equation*}
            \sum_{k=1}^n \sum_{\ell=1}^n \lambda_{k\ell} \, \rho(x_k,x_\ell) = \boldsymbol{1} (\bLambda \circ \bR) \boldsymbol{1}^\top \geq 0,
    \end{equation*}
    where $\boldsymbol{1}$ is an $n$-dimensional row vector of ones, $\bR=[\rho(x_k,x_\ell)]_{k,\ell=1}^n$, and $\circ$ is the Hadamard product. These inequalities hold true as soon as $\rho$ is a symmetric positive semidefinite function, which implies that $\bR$ is a symmetric positive semidefinite matrix and so is $\bLambda \circ \bR$ due to the Schur product theorem. Reciprocally, for $\bLambda=\blambda^\top\, \blambda$ with $\blambda\in \R^n$, it is seen that the gap inequalities (\ref{matrixgap}) imply that $\bR$ must be a positive semidefinite matrix. The sufficiency conditions are therefore equivalent to $\rho$ being a symmetric positive semidefinite function. To conclude the proof, we invoke the Daniell-Kolmogorov extension theorem, which ensures the existence of a zero-mean Gaussian random field on $\X$ with $\rho$ as its covariance function \citep[Theorem 3.1]{Doob1953}. \\

\noindent \textbf{Case 2: ${\cE}=\R_{\geq 0}$ (the case ${\cE}=\R_{\leq 0}$ is treated similarly)}. Let $x_1,\ldots,x_n$ be a set of points in $\X$ and $\rho$ a mapping satisfying the conditions of Theorem \ref{matrixgapineq}. Owing to Lemma \ref{lem:negeig1}, the gap inequalities (\ref{matrixgap}) are automatically fulfilled for all real symmetric matrices $\bLambda$, except those for which $\gamma(\bLambda,\R_{\geq 0})=0$, which correspond to the so-called \emph{copositive} matrices (Definition \ref{coposit}). This proves that the real symmetric matrix $\bR = [\rho(x_k,x_{\ell}]_{k,\ell=1}^n$ must belong to the cone of \emph{completely positive} matrices, which is the dual of the copositive cone in the vector space of real matrices endowed with the trace inner product \citep{Hall1963}. In particular, $\bR$ admits a factorization of the form \citep{Dannenberg2023}
    \begin{equation}
    \label{completelypositive}
        \bR = \sum_{j=1}^m \alpha_j \, \bz_j^\top \bz_j,
    \end{equation}
    with $m$ a positive integer, $\alpha_1,\ldots,\alpha_m$ non-negative real numbers summing to $1$, and $\bz_1,\ldots,\bz_m$ elements of $\R_{\geq 0}^n$. Therefore, $\bR$ is the non-centered covariance matrix of the random vector of $\R_{\geq 0}^n$ equal to $\bz_j$ with probability $\alpha_j$.

    Based on the fact that removing the last component of this random vector yields a reduced random vector with non-centered covariance matrix $ [\rho(x_k,x_{\ell}]_{k,\ell=1}^{n-1}$, one easily shows that the finite-dimensional distributions of the random vectors obtained for different values of $n$ and different choices of $x_1,\ldots,x_n$ in $\X$ are consistent. One can therefore invoke the Daniell-Kolmogorov extension theorem to assert that there exists a random field on $\X$ with values in $\R_{\geq 0}$ having $\rho$ as its non-centered covariance function.\\

\noindent \textbf{Case 3: ${\cE}$ is a compact subset of $\R$}. Let $\rho: \X \times \X \to \R$ be a mapping satisfying conditions 1 and 2 of the Theorem. Let ${\cC}(\Omega)$ be the set of all continuous functions on the sample space $\Omega = {\cE}^{\X}$. Once endowed with the supremum norm, ${\cC}(\Omega)$ is a normed vector space. Let ${\cG}$ be the set of functions of ${\cC}(\Omega)$ of the form $\omega \mapsto Z(x_k,\omega) \, Z(x_\ell,\omega)$, with $x_k$ and $x_\ell$ being points of $\X$ and $Z(\cdot,\omega)$ being an element of $\Omega$. Let ${\cM}$ be the linear subspace of ${\cC}(\Omega)$ spanned by ${\cG} \cup \{\mathsf{1}\}$, where $\mathsf{1}: \omega \mapsto 1$ is the constant function.

Define the linear operator $\E$ on ${\cM}$ as follows:
\begin{itemize}
    \item[(a)] $\E(\alpha_0 \cdot \mathsf{1} + \sum_{k=1}^n \alpha_k g_k) = \alpha_0+\sum_{k=1}^n \alpha_k \,\E(g_k)$ for any $g_1,\ldots,g_n \in {\cG}$ and $\alpha_0,\ldots,\alpha_n \in \R$.
    \item[(b)] $\E(Z(x_k,\cdot) \, Z(x_\ell,\cdot)) = \rho(x_k,x_\ell)$ for any $x_k, x_\ell \in \X$.
\end{itemize}
The latter condition is meaningful since $\rho$ is a symmetric mapping. The former condition implies, in particular, that $\E(\mathsf{1})=1$. It can be shown (Lemma \ref{lemma1} hereinafter) that the operator $\E$ is well defined on ${\cM}$.

Let $n$ be a positive integer, $\lambda_0$ a real number, $\bLambda=[\lambda_{k\ell}]_{k,\ell=1}^n$ a real square matrix, and $x_1,\ldots,x_n$ a set of points of $\X$. Suppose that the following inequality holds for every $\omega \in \Omega$:
\begin{equation*}
    \lambda_0 + \sum_{k=1}^n \sum_{\ell=1}^n \lambda_{k\ell} \, Z(x_k,\omega) \, Z(x_\ell,\omega) \geq 0.
\end{equation*}
Equivalently, $\lambda_0 \geq -\gamma(\bLambda,{\cE})$.
Provided that the gap inequalities (\ref{matrixgap}) are satisfied, one has:
\begin{equation*}
    -\gamma(\bLambda,{\cE}) + \sum_{k=1}^n \sum_{\ell=1}^n \lambda_{k\ell} \, \E(Z(x_k,\cdot) \, Z(x_\ell,\cdot)) \geq 0,
\end{equation*}
which implies
\begin{equation*}
    \E\left\{\lambda_0 \cdot \mathsf{1} + \sum_{k=1}^n \sum_{\ell=1}^n \lambda_{k\ell} \, Z(x_k,\cdot) \, Z(x_\ell,\cdot)\right\} = \lambda_0 + \sum_{k=1}^n \sum_{\ell=1}^n \lambda_{k\ell} \, \E(Z(x_k,\cdot) \, Z(x_\ell,\cdot)) \geq 0,
\end{equation*}
Accordingly, the following implication is true:
\begin{equation*}
    \lambda_0 \cdot \mathsf{1} + \sum_{k=1}^n \sum_{\ell=1}^n \lambda_{k\ell} \, Z(x_k,\cdot) \, Z(x_\ell,\cdot) \geq 0 \implies  \E\left\{\lambda_0 \cdot \mathsf{1} + \sum_{k=1}^n \sum_{\ell=1}^n \lambda_{k\ell} \, Z(x_k,\cdot) \, Z(x_\ell,\cdot)\right\} \geq 0,
\end{equation*}
i.e., the linear operator $\E$ is non-negative on $\cM$. This entails that it is norm-bounded \citep[Supplementary Material, claim 9.1.2]{Quintanilla2008}. Owing to the Hahn-Banach continuous extension theorem \citep[Theorem 2]{Buskes1993}, $\E$ can be extended to a norm-bounded linear non-negative operator on ${\cC}(\Omega)$.

According to Tychonoff's theorem, $\Omega={\cE}^{\X}$ is a compact space with respect to the product topology. Furthermore, since $\cE$ is a Hausdorff (aka separated) space, so is $\Omega$. We can therefore invoke the Riesz-Markov-Kakutani representation theorem \citep[Theorem 2.14]{Rudin1987} to assert that there exists a non-negative Borel measure $\P$ on $\Omega$ such that $\E(f) = \int_{\Omega} f(\omega) \P({\rm d}\omega)$ for every $f \in {\cC}(\Omega)$. This is a probability measure since $\E(\mathsf{1})=1$. For $f=Z(x_k,\cdot) Z(x_\ell,\cdot)$, one gets
\begin{equation*}
    \rho(x_k,x_\ell)=\E(Z(x_k,\cdot) Z(x_\ell,\cdot)) = \int_{\Omega} Z(x_k,\omega) Z(x_\ell,\omega) \P({\rm d}\omega),
\end{equation*}
i.e., there exists a random field $Z=\{Z(x,\omega): x\in \X, \omega \in \Omega\}$ valued in ${\cE}$ having $\rho$ as its non-centered covariance. \\

\noindent \textbf{Case 4: ${\cE}$ is a closed subset of $\R$}. 
Let ${\cal S}_n$ be the space of real square matrices of order $n$ endowed with the trace inner product, which is isomorphic to the Euclidean space $\R^{n^2}$ endowed with the usual scalar product. Let ${\cal P}_n$ be the set of matrices of the form $\bz \bz^\top$, with $\bz \in \cE^n$. Let ${\cal H}_n$ be the convex hull of ${\cal P}_n$, which is a closed set insofar as $\cE$ is closed. 

We first prove that a matrix $\bR_n$ belonging to ${\cal S}_n$ fulfills the gap inequalities (\ref{matrixgap}) if, and only if, it belongs to ${\cal H}_n$. On the one hand, owing to Carath\'eodory’s theorem, any element $\bR_n$ of ${\cal H}_n$ can be expressed as a convex combination of elements of ${\cal P}_n$. Equivalently:
\begin{equation}
\label{Pn}
    \bR_n = \int_{\cE^n} \bz \bz^\top \, \P_n ({\rm d}\bz),
\end{equation}
where $\P_n$ is a probability measure on ${\cE}^n$. By definition of the $\gamma$-gap, $\langle \bLambda, \bz \bz^\top \rangle \geq \gamma(\bLambda,\cE)$ for any $\bz \in \cE^n$ and $\bLambda \in {\cal S}_n$, hence, for any $\bLambda \in {\cal S}_n$,
\begin{equation}
\label{carath}    
\langle \bLambda, \bR_n \rangle = \int_{\cE^n} \langle \bLambda,\bz \bz^\top\rangle \, \P_n ({\rm d}\bz) \geq \int_{\cE^n} \gamma(\bLambda,\cE) \, \P_n ({\rm d}\bz) = \gamma(\bLambda,\cE),
\end{equation}
i.e, $\bR_n$ fulfills the gap inequalities. 

Reciprocally, for any $\bR_n \in {\cal S}_n \smallsetminus {\cal H}_n$, the hyperplane separation theorem \citep[Example 2.20]{Boyd2004} asserts that there exists a hyperplane that strictly separates $\bR_n$ and ${\cal H}_n$, i.e., there exists $\bLambda \in {\cal S}_n$ and $b \in \R$ such that $\langle \bLambda, \bR_n \rangle < b <  \langle \bLambda, \bB \rangle$ for all $\bB \in {\cal H}_n$. In particular, $b \leq \inf \left\{ \langle \bLambda, \bB \rangle: \bB \in {\cal P}_n \right\}= \gamma(\bLambda,\cE)$, so that $\langle \bLambda, \bR_n \rangle < \gamma(\bLambda,\cE)$, i.e., $\bR_n$ does not fulfill the gap inequalities.

According to (\ref{Pn}), the probability measure $\P_n$ characterizes the distribution of a random vector $\bZ$ of $\cE^n$ having $\bR_n$ as its non-centered covariance.

Finally, we prove that, if a mapping $\rho$ satisfies the conditions of Theorem \ref{matrixgapineq}, the probability measures $\P_n$ and $\P_{n-1}$ associated with $\bR_n = [\rho(x_k,x_\ell)]_{k,\ell=1}^n$ and $\bR_{n-1} = [\rho(x_k,x_\ell)]_{k,\ell=1}^{n-1}$, as defined in (\ref{Pn}), are consistent. This stems from the fact that $\bR_{n-1}$ is the orthogonal projection of $\bR_n$ onto ${\cal S}_{n-1}$, being obtained by removing the last row and last column of $\bR_n$, and that ${\cal H}_{n-1}$ is the orthogonal projection of ${\cal H}_n$ onto ${\cal S}_{n-1}$. Therefore, $\bR_n \in {\cal H}_n \Rightarrow \bR_{n-1} \in {\cal H}_{n-1}$. This translates into the fact that, given (\ref{Pn}), $\P_{n-1}$ is obtained by marginalizing $\P_n$ on $\cE^{n-1}$:
\begin{equation*}
    \bR_{n-1} = \int_{\cE^{n-1}} \bz \bz^\top \, \P_{n-1} ({\rm d}\bz), \text{ with } \ \P_{n-1}({\rm d}\bz) = \int_{z^\prime \in \cE} \, \P_{n} ({\rm d}(\bz,z^\prime)).
\end{equation*}

Thus, we can invoke the Daniell-Kolmogorov extension theorem to assert that there exists a random field $Z$ in $\X$ with values in $\cE$ having $\rho$ as its non-centered covariance function.\\
\end{proof}

\begin{lemma}
\label{lemma1}
   The linear operator $\E$ on $\cM$ defined by the above conditions (a) and (b) (case 3 in the proof of the sufficiency part of Theorem \ref{matrixgapineq}) is well defined.
\end{lemma}

\bigskip

\begin{proof}[Proof of Lemma \ref{lemma1}]
    We follow \citet[Supplementary Material, claim 9.1.1]{Quintanilla2008}. Let $g \in {\cM}$ and assume that it possesses two different representations as linear combinations of elements of ${\cG} \cup \{\mathsf{1}\}$:
    \begin{equation*}
        g = \alpha_0 \cdot \mathsf{1} + \sum_{k=1}^n \sum_{\ell=1}^n \alpha_{k\ell} \, Z(x_k,\cdot) \, Z(x_\ell,\cdot) = \alpha^\prime_0 \cdot \mathsf{1} + \sum_{k=1}^{n^\prime} \sum_{\ell=1}^{n^\prime} \alpha^\prime_{k\ell} \, Z(y_k,\cdot) \, Z(y_\ell,\cdot).
    \end{equation*}
    Then,
    \begin{equation*}
    \alpha_0 \cdot \mathsf{1} + \sum_{k=1}^n \sum_{\ell=1}^n \alpha_{k\ell} \, Z(x_k,\cdot) \, Z(x_\ell,\cdot) - \alpha^\prime_0 \cdot \mathsf{1} - \sum_{k=1}^{n^\prime} \sum_{\ell=1}^{n^\prime} \alpha^\prime_{k\ell} \, Z(y_k,\cdot) \, Z(y_\ell,\cdot) = g-g = 0\cdot \mathsf{1}
    \end{equation*}
    also belongs to ${\cM}$. Using the linearity of $\E$ on $\cM$, it comes
    \begin{equation*}
    \begin{split}
        \E &\Big\{ \alpha_0 \cdot \mathsf{1} + \sum_{k=1}^n \sum_{\ell=1}^n \alpha_{k\ell} \, Z(x_k,\cdot) \, Z(x_\ell,\cdot) \Big\} - \E \Big\{\alpha^\prime_0 \cdot \mathsf{1} + \sum_{k=1}^{n^\prime} \sum_{\ell=1}^{n^\prime} \alpha^\prime_{k\ell} \, Z(y_k,\cdot) \, Z(y_\ell,\cdot)\Big\} \\&= \E(0 \cdot \mathsf{1}) = 0 \cdot \E(\mathsf{1}) = 0,
    \end{split}
    \end{equation*}
    proving that any two representations of the same element of $\cM$ have the same image by $\E$.
\end{proof}

\bigskip

\begin{proof}[Proof of Theorem \ref{matrixgapineq_boundedset}] 
    Let $\bar{\cE}$ denote the closure of $\cE$ in $\R$ endowed with the usual topology. On the one hand, $\bar{\cE}$ is closed and bounded, hence it is a compact set of $\R$. On the other hand, one has
    $$\gamma(\bLambda,\cE) = \inf \{ \bz \bLambda \bz^\top: \bz \in {\cE}^n \} = \inf \{ \bz \bLambda \bz^\top: \bz \in {\bar{\cE}}^n \} = \gamma(\bLambda,\bar{\cE}).$$
    Accordingly, owing to Theorem \ref{matrixgapineq}, a symmetric mapping $\rho$ satisfying (\ref{matrixgap_boundedset}) is the non-centered covariance of a random field $Z$ on $\mathbb{X}$ valued in $\bar{\cE}$.

    Let $\epsilon > 0$. We define a $\epsilon$-neighborhood of $z \in \bar{\cE}$ as $N_\epsilon(z) = [z-\epsilon,z+\epsilon] \cap \cE$, and a mapping $\varphi_\epsilon: \bar{\cE} \to \cE$ that associates to each point of $\bar{\cE}$ a neighboring point of $\cE$:
    \begin{equation*}
        \varphi_\epsilon(z) = \dot{z}, \quad z \in \bar{\cE},
    \end{equation*}
    where $\dot{z}$ is an arbitrary point chosen in $N_\epsilon(z)$.

    From the above random field $Z$, we define the random field $\dot{Z} = \varphi_\epsilon(Z)$ valued in $\cE$ and the random field $D=\dot{Z}-Z$ valued in $[-\epsilon,\epsilon]$. The non-centered covariance of $\dot{Z}$ is
    \begin{equation*}
    \begin{split}
        \rho_\epsilon(x,y) &= \E(\dot{Z}(x,\cdot) \dot{Z}(y,\cdot)) \\
        &= \rho(x,y) + \E(D(x,\cdot) Z(y,\cdot)) + \E(Z(x,\cdot) D(y,\cdot)) + \E(D(x,\cdot) D(y,\cdot)).
    \end{split}
    \end{equation*}
    Because $Z$ is valued in the bounded set $\bar{\cE}$ and $\lvert D \rvert$ is upper-bounded by $\epsilon$, it is seen that $\rho_\epsilon$ tends pointwise to $\rho$ as $\epsilon$ tends to zero.\\

\end{proof}

\begin{proof}[Proof of Theorems \ref{variogr} and \ref{variogr2}]
    The theorems are clearly equivalent and are proved in a way similar to that of Theorem \ref{matrixgapineq}. In the sufficiency part (case 4), one just needs to replace the matrix $\bz \, \bz^\top$ by $[\boldsymbol{1}^\top \bz-\bz^\top \boldsymbol{1}]^2$ in the definition of ${\cal P}_n$, and the $\gamma$-gap by the $\eta$-gap.\\
\end{proof}

\begin{proof}[Proof of Example \ref{example2}]
    Let us show that $\rho_q$ fulfills the conditions of Theorem \ref{highorder}. The symmetry condition stems from the symmetry of the hafnian and of the mapping $\rho$ defining $\rho_q$.
    As for the gap inequalities, they stem from Theorem \ref{matrixgapineq} when $q=2$. Let us examine the case $q=4$. Let $Z$ be a random field on $\X$ valued in $\R$ with non-centered covariance $\rho$. Define a random field $Y$ on $\X \times \X$ as the product of two independent copies of $Z$. On account of Theorem \ref{matrixgapineq}, one can write
\begin{equation*}
    \begin{split}
        &\sum_{k_1=1}^n \sum_{k_2=1}^n \sum_{k_3=1}^n \sum_{k_4=1}^n \lambda_{k_1,k_2,k_3,k_4} \, \rho(x_{k_1},x_{k_2}) \rho(x_{k_3},x_{k_4}) \\
        &= \sum_{k_1=1}^n \sum_{k_2=1}^n \sum_{k_3=1}^n \sum_{k_4=1}^n \lambda_{k_1,k_2,k_3,k_4} \, \E(Y(x_1,x_3) \, Y(x_2,x_4)) \\
        &\geq \inf_{(z \, z^\prime) \ \in {\R}^2} \sum_{k_1=1}^n \sum_{k_2=1}^n \sum_{k_3=1}^n \sum_{k_4=1}^n \lambda_{k_1,k_2,k_3,k_4} \, z \, z^\prime \\
        &\geq \inf_{(z_{k_1},z_{k_2},z_{k_3},z_{k_4}) \ \in {\R}^4} \sum_{k_1=1}^n \sum_{k_2=1}^n \sum_{k_3=1}^n \sum_{k_4=1}^n \lambda_{k_1,k_2,k_3,k_4} \, z_{k_1} \, z_{k_2} \, z_{k_3} \, z_{k_4} = \gamma(\bLambda,\R),
    \end{split}
    \end{equation*}
with $\bLambda= [\lambda_{k_1,\ldots,k_4}]_{k_1,\ldots,k_4 = 1}^n$.
Accordingly, by definition of the hafnian,
\begin{equation*}
    \begin{split}
        &\sum_{k_1=1}^n \sum_{k_2=1}^n \sum_{k_3=1}^n \sum_{k_4=1}^n \lambda_{k_1,k_2,k_3,k_4} \, \rho_4(x_{k_1},x_{k_2},x_{k_3},x_{k_4}) \geq 3 \, \gamma(\bLambda,\R).
    \end{split}
    \end{equation*}
    Now, $\gamma(\bLambda,\R)$ is either $0$ or $-\infty$. Indeed, $\sum_{k_1=1}^n \sum_{k_2=1}^n \sum_{k_3=1}^n \sum_{k_4=1}^n \lambda_{k_1,k_2,k_3,k_4} \, z_{k_1} z_{k_2} z_{k_3} z_{k_4}$ is zero when $(z_{k_1},z_{k_2},z_{k_3},z_{k_4})=(0,0,0,0)$ and, if the quadruple sum is negative for some $(z_{k_1},z_{k_2},z_{k_3},z_{k_4})\neq (0,0,0,0)$, then it tends to $-\infty$ for $(az_{k_1},az_{k_2},az_{k_3},az_{k_4})$ with $a$ tending to infinity. This implies that $3 \, \gamma(\bLambda,\R)= \gamma(\bLambda,\R)$ and concludes the proof for $q=4$. The proof for $q>4$ can be done similarly by induction on the product space on which the random field $Y$ is defined.

\end{proof}

\begin{proof}[Proof of Theorem \ref{functiongapineq}]

First note that $L^2(\X^2,\mu)$ is a Hilbert space \citep[Example 4.5(b)]{Rudin1987}; in particular, being a complete normed space, it is locally convex. Furthermore, owing to the Riesz representation theorem \citep[Theorem 13.32]{Roman2008}, it is self-dual, i.e., isometrically isomorphic to its dual space.

Let ${\cal P}$ be the set of functions in $\X^2$ of the form $(x,y) \mapsto \varphi_z(x,y) = z(x) z(y)$, with $z \in \cE^\X$. Let ${\cal H}$ be the closed
convex hull of ${\cal P}$. Since $\cE$ is compact, Tychonoff's theorem asserts that $\cE^{\X}$ is compact with respect to the product topology, and so are
${\cal P}$ and, based on Theorem 3.20(c) of \cite{Rudin1991}, ${\cal H}$. 

We prove that a function $\rho$ belonging to $L^2(\X^2,\mu)$ fulfills the conditions of Theorem \ref{functiongapineq} if, and only if, it belongs to ${\cal H}$. On the one hand, Choquet's theorem \citep[Chapter 3]{Phelps2001} and Milman's theorem \citep[Theorem 3.25]{Rudin1991} assert that any element $\rho$ of ${\cal H}$ can be expressed as a convex combination of elements of ${\cal P}$. Therefore, $\rho$ is symmetric and such that
\begin{equation}
\label{Pn2}
    \rho(x,y) =  \int_{\cE^\X} \varphi_z(x,y) \, \P ({\rm d}z)= \int_{\cE^\X} z(x) z(y) \, \P ({\rm d}z), \quad x,y \in \X,
\end{equation}
where $\P$ is a probability measure on ${\cE}^\X$. Since, furthermore, $\mu$ is finite and $\cE$ is bounded, any function $\varphi_z$ in (\ref{Pn2}) belongs to $ L^2(\X^2,\mu)$. By definition of the $\gamma$-gap, $\langle \lambda, \varphi_z \rangle_{\mu} \geq \gamma(\lambda,\cE,\mu)$ for any $z \in \cE^\X$ and $\lambda \in L^2(\X^2,\mu)$, hence, for any $\lambda \in L^2(\X^2,\mu)$, $$\langle \lambda, \rho \rangle_{\mu} = \int_{\cE^\X} \langle \lambda, \varphi_z \rangle_{\mu} \, \P ({\rm d}z) \geq \int_{\cE^\X} \gamma(\lambda,\cE,\mu) \, \P ({\rm d}z) = \gamma(\lambda,\cE,\mu),$$ i.e, $\rho$ fulfills the gap inequalities. Reciprocally, for any $\rho \in L^2(\X^2,\mu) \smallsetminus {\cal H}$, the Hahn–Banach separation theorem \citep[Theorem 3.4(b)]{Rudin1991} asserts that there exists a hyperplane that strictly separates $\rho$ and ${\cal H}$, i.e., there exists $\lambda \in L^2(\X^2,\mu)$ and $b \in \R$ such that $\langle \lambda, \rho \rangle_{\mu} < b <  \langle \lambda, f \rangle_{\mu}$ for all $f \in {\cal H}$. In particular, as ${\cal P} \subset {\cal H}$, $b \leq \inf \left\{ \langle \lambda, f \rangle_{\mu}: f \in {\cal P} \right\}= \gamma(\lambda,\cE,\mu)$, so that $\langle \lambda, \rho \rangle_{\mu} < \gamma(\lambda,\cE,\mu)$, i.e., $\rho$ does not fulfill the gap inequalities.

Accordingly, a function $\rho$ in $L^2(\X^2,\mu)$ fulfills the conditions of Theorem \ref{functiongapineq} if, and only if, it admits a representation of the form (\ref{Pn2}), i.e., if, and only if, it is the non-centered covariance of a random field on $\X$ with values in $\cE$ (the distribution of which is characterized by the probability measure $\P$).

To conclude the proof, note that the antisymmetric part of $\lambda \, {\rm d}\mu$ does not contribute to the integral in (\ref{functiongap}) because of the symmetry of $\rho$, therefore only the symmetric part of $\lambda \, {\rm d}\mu$ matters.

\end{proof}

\begin{proof}[Proof of Theorem \ref{functiongapineqRplus}]
    \noindent \textit{Necessity.} Let $Z$ be a random field on $\X$ with values in $\cE$ and non-centered covariance function $\rho$. On the one hand, the latter function is symmetric. On the other hand, by definition of the $\gamma$-gap, one has, for any continuous function $\lambda$,
    \begin{equation}
    \label{R+}
        \int_{\X^2} \lambda(x,y) \, Z(x,\cdot) \, Z(y,\cdot) \, {\rm d}\mu(x,y) \geq \gamma(\lambda,\cE,\mu),
    \end{equation}
    where the quantities on both sides may be infinite. The gap inequality (\ref{functiongapRplus}) follows by taking the expected value of both sides of (\ref{R+}).

    \medskip

    \noindent \textit{Sufficiency.} We develop the proof for the case $\cE = \R_{\geq 0}$; the remaining cases can be proved in the same way. Let $\rho: \X \times \X \to \R$ be a symmetric continuous function that is not the non-centered covariance of any random field on $\X$ with non-negative values. According to Theorem \ref{matrixgapineq}, there exists an integer $n$, a set of points $x_1,\ldots, x_n$ and a copositive matrix $\bLambda = [\lambda_{k\ell}]_{k,\ell=1}^n$ such that the discrete gap inequality (\ref{matrixgap}) does not hold:
    \begin{equation*}
        \sum_{k=1}^n \sum_{\ell=1}^n \lambda_{k\ell} \, \rho(x_k,x_\ell) < 0.
    \end{equation*}
    Since $\varpi$ is positive and $\rho$ is continuous, we can find a ``small enough'' open ball $O$ centered at the origin of $\X^2$ such that
    \begin{equation*}
        \int_{\X^2} \lambda(x,y) \, \rho(x,y) \, {\rm d}\varpi(x) \, {\rm d}\varpi(y) < 0,
    \end{equation*}
    where
    \begin{itemize}
    \item $O+(x_k,x_\ell) \subset \X$ for all $k, \ell = 1,\ldots,n$
    \item $\{O+(x_k,x_\ell): k,\ell=1,\ldots,n \}$ are pairwise disjoint
    \item $\lambda$ is a function in $\X^2$ defined by
    $$\lambda(x,y) = \begin{cases}
            \lambda_{k\ell} \text{ if $(x-x_k,y-x_\ell) \in O$}\\
            0 \text{ otherwise.}
        \end{cases}$$
    \end{itemize}
    For this particular function $\lambda$ and any function $z$ defined on $\X$, one has:
    \begin{equation*}
        \int_{\X^2} \lambda(x,y) \, z(x) \, z(y) \, {\rm d}\varpi(x) \, {\rm d}\varpi(y) =  \int_O \sum_{k=1}^n \sum_{\ell=1}^n \lambda_{k \ell} \, z(x+x_k) \, z(y+x_\ell) \, {\rm d}\varpi(x+x_k) \, {\rm d}\varpi(y+x_\ell),
    \end{equation*}
    where the double sum in the integrand is non-negative since $\bLambda$ is copositive. Accordingly, the gap $\gamma(\lambda,\R_{\geq 0},\mu)$ is non-negative and
    \begin{equation*}
        \int_{\X^2} \lambda(x,y) \, \rho(x,y) \, {\rm d}\varpi(x) \, {\rm d}\varpi(y) < \gamma(\lambda,\R_{\geq 0},\mu),
    \end{equation*}
    which proves that $\rho$ does not fulfill the gap inequality (\ref{functiongapRplus}) for the function $\lambda$ defined above.

\end{proof}

\begin{proof}[Proof of Theorem \ref{th:PSD}]
    Given that $\gamma(\bLambda,\Z)=\gamma(\bLambda,\R)$ for any real symmetric matrix $\bLambda$ (Corollary \ref{lem:gammaR}), the conditions of Theorem \ref{th:PSD} are equivalent to that of Theorem \ref{matrixgapineq} when ${\cE} = \R$ or $\Z$, as shown in the proof of the sufficiency part (case 1) of the latter theorem.\\
\end{proof}

\begin{proof}[Proof of Theorem \ref{variogr_unboundedset}]
We establish the equivalence of Theorem \ref{variogr_unboundedset} with Theorem \ref{variogr2} in the case when $\cE = \R$ or $\Z$. Let $\bLambda=[\lambda_{k\ell}]_{k,\ell=1}^n$ be a real symmetric matrix with zero diagonal entries. Lemma \ref{lem:etaR} ensures that $\eta(\bLambda,\cE)$ can take only two values:
    \begin{enumerate}
        \item $\eta(\bLambda,\cE) = +\infty$: if so, the gap inequality (\ref{gap4variog}) is automatically fulfilled.
        \item $\eta(\bLambda,\cE) = 0$. On account of Lemma \ref{gamma2eta}, this happens if, and only if, the matrix $\bLambda - \bDelta$ is positive semidefinite. Accounting for the fact that $g$ is symmetric and equal to zero on the diagonal of $\X \times \X$, one has:
    \begin{equation*}
        \langle \bLambda, \bG \rangle = \sum_{k=1}^n \sum_{\ell=1}^n \lambda_{k \ell} \, g(x_{k},x_{\ell}) = \langle \bDelta-\bLambda, \widetilde{\bG} \rangle,
    \end{equation*}
    with $\bG = [g(x_k,x_\ell)]_{k,\ell=1}^n$ and $\widetilde{\bG} = [g(x_1,x_\ell)+g(x_k,x_1)-g(x_k,x_\ell)]_{k,\ell=1}^n$. Therefore, for the matrices $\bLambda$ such that $\eta(\bLambda,\cE) = 0$, the gap inequality can be rewritten as
    \begin{equation}
    \label{traceineq}
        0 \leq \langle \bLambda, \bG \rangle=\langle \bDelta-\bLambda, \widetilde{\bG} \rangle= \boldsymbol{1} ((\bLambda -\bDelta) \circ \widetilde{\bG}) \boldsymbol{1}^\top,
    \end{equation}
    where $\bLambda-\bDelta$ can be any symmetric positive semidefinite matrix for which every column and every row sums zero. Taking $\bLambda = \blambda \, \blambda^\top$ where $\blambda$ is a vector of $\R^n$ whose elements sum to zero, it is seen that $g$ must fulfill the conditional negative semidefiniteness condition (\ref{condneg}). Reciprocally, if $g$ is conditionally negative semidefinite, then $\widetilde{\bG}$ is a positive semidefinite matrix \citep[Lemma 2.4]{Reams1999} and so is $(\bLambda -\bDelta) \circ \widetilde{\bG}$ due to the Schur product theorem, so that the inequality (\ref{traceineq}) holds.
        \end{enumerate}

It is concluded that, when ${\cE} = \R$ or $\Z$, condition 3 of Theorem \ref{variogr_unboundedset} is equivalent to condition 3 of Theorem \ref{variogr2}, therefore the both theorems are equivalent. \\   
\end{proof}

\begin{proof}[Proof of Theorem \ref{th:PSDZ0}]
    Let $\bLambda$ be a real symmetric matrix of order $n$. If $\bLambda$ has a negative eigenvalue, then $\gamma(\bLambda,\Z \smallsetminus \{0\}) = -\infty$ (Lemma \ref{lem:negeig}) and the gap inequalities (\ref{matrixgap}) are automatically fulfilled. Otherwise, $\bLambda$ is positive semidefinite and one has, for any real symmetric positive semidefinite matrix $\bR$ of order $n$:
    \begin{equation*}
        \begin{split}
            \langle \bLambda, \bR + \varepsilon \boldsymbol{I} \rangle &= \langle \bLambda, \bR \rangle + \varepsilon \, \text{tr}(\boldsymbol{\bLambda})\\ & \geq \boldsymbol{1} (\bLambda \circ \bR) \boldsymbol{1}^\top + \varepsilon \, n \, \text{det}(\boldsymbol{\bLambda})^{\frac{1}{n}} \text{ (AM-GM inequality)}\\
            & \geq 0 +            \frac{\varepsilon \, n }{\gamma_n} \gamma(\bLambda,\Z\smallsetminus\{0\}) \text{ (Lemma \ref{lem:Nplus})}\\
            & \geq \gamma(\bLambda,\Z\smallsetminus\{0\}) \text{ (Lemma \ref{lem:Nplus})}.
        \end{split}
    \end{equation*}
    Accordingly, $\rho+\varepsilon \, \delta$ satisfies the sufficient conditions of Theorem \ref{matrixgapineq} for $\cE=\Z \smallsetminus \{0\}$ and is therefore the non-centered covariance of a random field valued in $\Z \smallsetminus \{0\}$.

    If $\rho(x,x) \geq \frac{1}{3}$ for any $x \in \X$ and $\varepsilon \geq \frac{2}{3}$, then the gap inequalities are trivially satisfied for any positive semidefinite matrix $\bLambda$ of order $1$ (read: any non-negative real value $\lambda_{11}$), so the above proof can be limited to the case $n \geq 2$, for which the lower bound on $\varepsilon$ can be reduced to $\frac{2}{3}$ owing to Lemma \ref{lem:Nplus}.
\end{proof}

\begin{proof}[Proof of Theorem \ref{th:CP}]
    \noindent \textit{Necessity.} Let $\rho$ be the non-centered covariance of a unit process in $\X$, $\bLambda=[\lambda_{k\ell}]_{k,\ell=1}^n$ be a corner positive matrix, and $x_1,\ldots,x_n$ be a set of points in $\X$. Then, Conditions 1 and 3 of Theorem \ref{th:CP} are derived in a straightforward manner from Theorem \ref{matrixgapineq} and Definition \ref{def:CP}. Furthermore, the gap inequalities (\ref{matrixgap}) applied with a $1 \times 1$ matrix $\bLambda=\lambda_{11}$ gives $\rho(x,x) \geq 1$ when choosing $\lambda_{11}=1$ and $\rho(x,x) \leq 1$ when choosing $\lambda_{11}=-1$, which leads to the remaining condition $\rho(x,x)=1$ for any $x \in \X$.\\

    \noindent \textit{Sufficiency.} Suppose that $\rho$ is a symmetric mapping in $\X \times \X$ fulfilling Eq. (\ref{corner}) and such that $\rho(x,x)=1$ for any $x \in \X$. Let $\bLambda=[\lambda_{k\ell}]_{k,\ell=1}^n$ be a real square matrix and $x_1,\ldots,x_n$ be a set of points in $\X$. Denote by $\bJ$ the matrix of size $n \times n$ with all its entries equal to $0$, except the entry in the first row and first column that is equal to $1$. Then, $\bLambda+\gamma(\bLambda,\{-1,1\}) \cdot \bJ$ is corner positive, and the application of (\ref{corner}) leads to the gap inequalities (\ref{matrixgap}).\\
\end{proof}

\begin{proof}[Proof of Theorem \ref{thMcMillan}]
A constructive proof is given by \cite{McMillan1955} for $\X=\Z$. We can offer a simpler alternative proof based on Theorem \ref{matrixgapineq}. As $\rho$, $\rho^*$ is a symmetric mapping and it therefore remains to prove that the gap inequalities (\ref{matrixgap}) hold for any real symmetric matrix $\bLambda$. The clue is to decompose $\bLambda$ into a matrix with zero diagonal entries $\bar{\bLambda}$ and a diagonal matrix $diag(\bLambda)$ and to notice that the gaps $\gamma(\bar{\bLambda},[-1,1])$ and $\gamma(\bar{\bLambda},\{-1,1\})$ are the same \citep[Lemma 2.2]{Megretski2001}. Accordingly, for any set of points $x_1,\ldots, x_n$,
\begin{equation*}
\begin{split}
    \sum_{k=1}^n \sum_{\ell=1}^n \lambda_{k,\ell} \, \rho^*(x_k,x_\ell) &= \sum_{k=1}^n \sum_{\ell=1}^n \bar{\lambda}_{k,\ell} \, \rho(x_k,x_\ell) + \sum_{k=1}^n \lambda_k \\&\geq \gamma(\bar{\bLambda},\{-1,1\}) + \text{tr}(\bLambda) \\&= \gamma({\bLambda},\{-1,1\}),
\end{split}
\end{equation*}
which concludes the proof. \\
\end{proof}

\begin{proof}[Proof of Theorem \ref{th:necessary-11}]
    For $\bLambda(u,v)=[\lambda_{k\ell}(u,v)]_{k,\ell=1}^n$ and $\bz \in \{-1,1\}^n$, one has
\begin{equation*}
    \bz \bLambda(u,v) \bz^\top = \left[\sum_{k=1}^n (-1)^{\sum_{j=1}^m b_j(k) b_j(u)} z_k \right] \left[\sum_{\ell=1}^n (-1)^{\sum_{j=1}^m b_j(\ell) b_j(v)} z_\ell \right] := (\boldsymbol{U}^\top \boldsymbol{1}) (\boldsymbol{1}^\top \boldsymbol{V}),
\end{equation*}
where $\boldsymbol{U}$ and $\boldsymbol{V}$ are the $n$-dimensional vectors whose entries are the summands in the above expression. The $i$-th entries of $\boldsymbol{U}$ and $\boldsymbol{V}$ are the same unless the number of bit flips between $\boldsymbol{b}(i) \circ \boldsymbol{b}(u)$ and $\boldsymbol{b}(i) \circ \boldsymbol{b}(v)$ is odd, this number being $r_i(u,v)=(\boldsymbol{b}(u) \veebar \boldsymbol{b}(v)) \, \boldsymbol{b}(i)^\top$. Accordingly, up to a reordering of the entries of $\boldsymbol{U}$ and $\boldsymbol{V}$, one can split these vectors into $\boldsymbol{U}=[\boldsymbol{W}_1 ,\boldsymbol{W}_2]$ and $\boldsymbol{V}=[\boldsymbol{W}_1,-\boldsymbol{W}_2]$, where $\boldsymbol{W}_1 \in \{-1,1\}^{n-q(u,v)}$, $\boldsymbol{W}_2 \in \{-1,1\}^{q(u,v)}$ and $q(u,v)$ is the number of odd entries in $[r_i(u,v)]_{i=1}^n$. This entails
\begin{equation*}
    \bz \bLambda(u,v) \bz^\top = \sigma_1^2 - \sigma_2^2,
\end{equation*}
with $\sigma_1$ and $\sigma_2$ the sums of the entries of $\boldsymbol{W}_1$ and $\boldsymbol{W}_2$, respectively. Since these vectors can be any element of $\{-1,1\}^{n-q(u,v)}$ and $\{-1,1\}^{q(u,v)}$ (because $\bz$ can be any element of $\{-1,1\}^n$), the minimal value of $\bz \bLambda(u,v) \bz^\top$---that is, the gap $\gamma(\bLambda(u,v),\{-1,1\})$---is obtained when
\begin{itemize}
    \item $\sigma_1 = 0$, which is realizable only if $n-q(u,v)$ is even, or $|\sigma_1| = 1$, realizable if $n-q(u,v)$ is odd; that is, $\sigma_1 = b_m(n-q(u,v))$, irrespective of the parity of $n-q(u,v)$;
    \item $|\sigma_2| = q(u,v)$, which can always be attained.
\end{itemize}
Since $q(u,v)$ can be expressed as indicated in the claim of the Theorem, one concludes the proof by invoking Theorem \ref{matrixgapineq} (necessary part).
\end{proof}

\begin{proof}[Proof of Theorem \ref{variogr_unitprocess}]
    The result stems from Theorem \ref{variogr} and the fact that, for a unit process $Z$ with variogram $g$, $Z-\E(Z)$ has no drift and the same semivariogram $g$. Alternatively, it also stems from Theorem \ref{matrixgapineq} and the relationship $\rho=1-g$ between the non-centered covariance $\rho$ and the semivariogram $g$ of a unit process.\\
\end{proof}

\begin{proof}[Proof of Theorem \ref{th:boundedinterval}]
    \noindent \textit{Necessity.} Any covariance function is symmetric and positive semidefinite. Moreover, if a random field $Z$ is valued in $[-1,1]$, then so does the product $Z(x_k,\cdot) Z(x_{\ell},\cdot)$ of any two of its components and, by taking the expected value, the non-centered covariance $\rho(x_k,x_{\ell})$.\\

    \noindent \textit{Sufficiency.} The conditions of the theorem would be sufficient if they implied the gap inequalities (\ref{matrixgap}), which are equivalent to:
    \begin{equation*}
        \langle \bLambda, \bR \rangle \leq \sup \{\bz \bLambda \bz^\top: \bz \in [-1,1]^n\}
    \end{equation*}
    for any real symmetric matrix $\bLambda$ of order $n$ and any $\bR=[\rho(x_k,x_\ell)]_{k,\ell=1}^n$ where $x_1,\ldots,x_n$ are points in $\X$. However, this inequality cannot be satisfied for any such $\bLambda$ \citep[Section 2.3.1]{Megretski2001}. Restricting $\rho$ to be valued in a smaller interval of the form $[-a,a]$ would not suffice either.
\end{proof}

\begin{proof}[Proof of Theorem \ref{th:boundedinterval20}]
\noindent \textit{Necessity.} Let $Z$ be a random field on $\X$ with non-centered covariance $\rho$ and values in $[-1,1]$. For $x \in \X$ and $u \in [-1,1]$, let $\mathsf{1}_{Z(x,\cdot) > u}$ be the indicator random variable equal to $1$ if $Z(x,\cdot) \geq u$, $0$ otherwise. One has:
\begin{equation}
\label{sumind1}
        \int_{-1}^{1} \mathsf{1}_{Z(x,\cdot) > u} \, {\rm d}u = \int_0^{2} \mathsf{1}_{Z(x,\cdot)+1 > u} \, {\rm d}u = Z(x,\cdot)+1,
\end{equation}
and, by taking the expected values of both sides,
\begin{equation}
\label{sumind2}
    \int_{-1}^{1} \E(\mathsf{1}_{Z(x,\cdot) > u})\, {\rm d}u = \E(Z(x,\cdot))+1.
\end{equation}
From (\ref{sumind1}), it comes:
\begin{equation*}
\begin{split}
        \int_{-1}^{1} \int_{-1}^{1} \E\left( \mathsf{1}_{Z(x,\cdot) > u} \, \mathsf{1}_{Z(y,\cdot) > v} \right) \, {\rm d}u \, {\rm d}v &= \E\left((Z(x,\cdot)+1) \, (Z(y,\cdot)+1) \right) \\&= \rho(x,y)+\E(Z(x,\cdot))+\E(Z(y,\cdot))+1.
\end{split}
\end{equation*}
Arguments in \citet[Appendix C]{emery2025} imply the following identity:
\begin{equation*}
    \frac{1}{2} \E\left( [\mathsf{1}_{Z(x,\cdot) > u} - \mathsf{1}_{Z(y,\cdot) > v}]^2 \right) = \frac{1}{2\pi} \int_0^1 \arccos C_{u,v;t}(x,y) \, {\rm d}t,
\end{equation*}
where, for each $t \in [0,1]$, $\{C_{u,v;t}: u, v \in [-1,1] \}$ are the cross-covariances of a family of jointly Gaussian random fields $\{Z_{u;t}: u \in [-1,1]\}$ with zero means and unit variances. Equivalently, one can view these random fields as a single standard Gaussian random field defined on $\X \times [-1,1]$ and write
\begin{equation}
\label{sumind3}
    \frac{1}{2} \E\left( [\mathsf{1}_{Z(x,\cdot) > u} - \mathsf{1}_{Z(y,\cdot) > v}]^2 \right) = \frac{1}{2\pi} \int_0^1 \arccos C_{t}((x,u),(y,v)) \, {\rm d}t,
\end{equation}
where, for each $t \in [0,1]$, $C_{t}$ is the covariance of a standard Gaussian random field $Z_{t}$ on $\X\times [-1,1]$, i.e., $C_{t}$ is symmetric, positive semidefinite and equal to $1$ on the diagonal of $(\X\times [-1,1])^2$. Using (\ref{sumind1}) to (\ref{sumind3}), one finds
\begin{equation*}
        \rho(x,y) = 1-\frac{1}{2\pi} \int_{-1}^{1} \int_{-1}^{1} \int_0^1 \arccos C_{t}((x,u),(y,v)) \, {\rm d}t \, {\rm d}u \, {\rm d}v,
\end{equation*}
which is the same as (\ref{sufficientrho10}).\\

\noindent \textit{Sufficiency.} Suppose that $\rho$ is given by (\ref{sufficientrho10}). Owing to the Daniell-Kolmogorov theorem, for every $t \in [0,1]$, there exists a standard Gaussian random field $Z_t$ on $\X \times [-1,1]$ having covariance $C_t$. The centered covariance of the median indicator of this random field, $\mathsf{1}_{Z_t>0}$, is \citep{McMillan1955}
\begin{equation*}
    \rho_t(x,u,y,v)=\frac{1}{2\pi} \arcsin C_t((x,u),(y,v)),
\end{equation*}
which is also the covariance of the centered indicator $\mathsf{1}_{Z_t>0}-\frac{1}{2}$. Let $Y_t$ be the random field on $\X$ defined by
$$Y_t(x,\cdot) = \int_{-1}^1 \left(\mathsf{1}_{Z_t((x,u),\cdot)>0}-\frac{1}{2}\right) \, {\rm d}u, \quad x \in \X.$$
From the representation (\ref{sufficientrho10}), it is seen that $\rho$ is the centered covariance of the random field $Y_T$, where $T$ is an independent random variable uniformly distributed in $[0,1]$. Such a random field is valued in  $[-1,1]$ and has a zero mean, therefore $\rho$ is also its non-centered covariance.

\end{proof}

\begin{proof}[Proof of Theorem \ref{complpos}]
See the proof of Theorem \ref{matrixgapineq}, in particular case 2 of the sufficiency part.\\
\end{proof}

\begin{proof}[Proof of Theorem \ref{Rplus1}]
    \noindent \textit{Necessity.} Any covariance function is symmetric and positive semidefinite. Moreover, if a random field takes non-negative values, then so does its non-centered covariance function.\medskip

    \noindent \textit{Sufficiency.} Let $x_1,\ldots,x_n$ be a set of points in $\X$ and $\rho$ a mapping satisfying the conditions of Theorem \ref{Rplus1}. Being symmetric and doubly non-negative, the matrix $\bR = [\rho(x_k,x_{\ell}]_{k,\ell=1}^n$ is completely positive if $n \leq 4$ \citep{Maxfield1962}. In such a case, $\rho$ is the non-centered covariance of the random field on $\X$ (see the proof of Theorem \ref{matrixgapineq}).
    In contrast, if $n \geq 5$, it has been shown that a doubly non-negative matrix may not be completely positive \citep{Burer2009, Strekelj2025} and therefore may not be factorizable as in (\ref{completelypositive}), i.e., there may be no random field on $\X$ having $\rho$ as its non-centered covariance.\\
\end{proof}

\begin{proof}[Proof of Theorem \ref{R0}]
Let $C: \X \times \X \to \R$ be an arbitrary symmetric positive semidefinite function. The Daniell-Kolmogorov extension theorem guarantees the existence of a zero-mean Gaussian random field $Z$ on $\X$ with covariance $C$. Using formula (A.23) of \cite{chiles_delfiner_2012}, it is seen that $\rho=\exp(C)$ is the non-centered covariance of the lognormal random field $Y$ defined by
$$Y(x,\cdot)=\exp\left(Z(x,\cdot)-\frac{C(x,x)}{2}\right), \quad x \in \X,$$ which is valued in $\R_{>0}$. Accordingly, $\rho$ is symmetric and positive and $\ln(\rho)=C$ can be any symmetric positive semidefinite function.\\
\end{proof}

\bibliography{bibcodomain}

\end{document}